\documentclass{siamltex}
\usepackage{amssymb}
\usepackage{xcolor}

\begin{document}

\title{Stable Well-posedness and  Tilt stability  with respect to admissible functions\thanks{This research was supported by the
National Natural Science Foundation of People's Republic of China (Grant No. 11371312) and IRTSTY.}}

\author{xi yin Zheng\thanks{Department of Mathematics, Yunnan University, Kunming 650091, People's Republic of China (xyzheng@ynu.edu.cn, jiangxingzhu@yahoo.com).}\and Jiangxing Zhu$^\dag$}

\maketitle

\begin{abstract}
Note that the well-posedness of a proper lower semicontinuous function $f$ can be equivalently described using an admissible function. In the case  when the objective function $f$  undergos the tilt perturbations in the sense of Poliquin and Rockafellar,  adopting  admissible functions $\varphi$ and $\psi$, this paper introduces and studies the stable well-posedness  of $f$ with respect to $\varphi$ (in breif, $\varphi$-SLWP) and  tilt-stable local minimum of $f$ with respect to $\psi$ (in brief, $\psi$-TSLM). In the special case when $\varphi(t)=t^2$ and $\psi(t)=t$, the corresponding $\varphi$-SLWP and $\psi$-TSLM reduce to the stable second order local minimizer and tilt stable local minimum respectively, which have been extensively studied in recent years. We discover an interesting relationship between two admissible functions $\varphi$ and $\psi$: $\psi(t)=(\varphi')^{-1}(t)$, which implies that a proper lower semicontinous function $f$ on a Banach space has $\varphi$-SLWP if and only if  $f$ has  $\psi$-TSLM. Using the techniques of variational analysis and conjugate analysis, we also prove that the strong  metric $\varphi'$-regularity of $\partial f$ is a sufficient condition for  $f$ to have $\varphi$-SLWP  and that the strong metric $\varphi'$-regularity of $\partial\overline{\rm co}(f+\delta_{B[\bar x,r]})$ for some $r>0$ is a necessary condition for  $f$ to have $\varphi$-SLWP.  In the special case when $\varphi(t)=t^2$, our results cover some existing main results  on the tilt stability.
\end{abstract}



\pagestyle{myheadings}
\thispagestyle{plain}
\markboth{XI YIN ZHENG AND JIANGXING ZHU}{Stable well-posedness and tilt stability}

\section{Introduction}

Well-posedness is a fundamental notion in variational analysis and optimization theory and has been well studied (cf. \cite{wellposed optimization problems,HY,L,R,yao and zheng 2014} and the references therein). Let $f$ be a proper lower semicontinuous function on a Banach space $X$ and recall that $f$ is well-posed at $\bar x\in{\rm dom}(f)$ (in the Tykhonov sense) if every minimizing sequence $\{x_n\}$ of $f$ converges to $\bar x$. Clearly, the  well-posedness of $f$ at $\bar x$ implies that
$\mathop{\arg\min}_{x\in X} f(x)=\{\bar x\}$. In the  case that $\mathop{\arg\min}_{x\in X} f(x)$ is not a singleton, we can adopt the following weak (or generalized) well-posedness:
$$d\big(x_n,\mathop{\arg\min}_{z\in X} f(z)\big):=\inf\big\{\|x_n-x\|:\;x\in \mathop{\arg\min}_{z\in X} f(z)\big\}\rightarrow0$$
for every minimizing sequence $\{x_n\}$ of $f$. We note that  well-conditionedness, Levitin-Polyak well-posedness, Hadamard well-posedness and other concepts are closely related or essentially equivalent to the above well-posedness and weak well-posedness (cf. \cite{AW, wellposed optimization problems,L}). Recall that  $\varphi:\mathbb{R}_+\rightarrow\mathbb{R}_+$ is an  admissible function if it is a nondecreasing function such that $\varphi(0)=0$ and  $[\varphi(t)\rightarrow0\Rightarrow t\rightarrow0]$. Some authors named  an admissible function as a forcing function, conditioning function and so on (cf. \cite{AW, bedn04, wellposed optimization problems}). It is known (cf.\cite[P6, Theorem 12]{wellposed optimization problems}) that $f$ is well-posed  at $\bar x$  if and only if  there exists an admissible function $\varphi:\mathbb{R}_+\rightarrow\mathbb{R}_+$ such that
$$\varphi(\|x-\bar x\|)\leq f(x)-f(\bar x)\quad\forall x\in X;\leqno{\rm (WP)}$$
while $f$ has weak well-posedness at $\bar x$ if and only if there exists an admissible function $\varphi:\mathbb{R}_+\rightarrow\mathbb{R}_+$ such that
$$\varphi(d(x,\mathop{\arg\min}_{x\in X} f(x)))\leq f(x)-f(\bar x)\quad\forall x\in X.\leqno{\rm (GWP)}$$
Replacing the entire space $X$ with some open ball $B_X(\bar x,r)$, one can consider the following respective localization of (WP) and (GWP):
$$\varphi(\|x-\bar x\|)\leq f(x)-f(\bar x)\quad\forall x\in B_X(\bar x,r)\leqno{\rm (LWP)}$$
and
$$\varphi\left(d\left(x,\mathop{\arg\min}_{x\in B_X(\bar x,r)} f(x)\right)\right)\leq f(x)-f(\bar x)\quad\forall x\in B(\bar x,r). \leqno{\rm (LGWP)}$$
In Attouch and Wets \cite{AW}, $\bar x$ is called a $\varphi$-minimizer of $f$ if (LWP) holds.
In the case that $\varphi(t)=ct$ with $c$ being a positive constant, (LWP) and (LGWP) reduce respectively to Polyak's sharp minimizer and Ferris' weak sharp minimizer which have been extensively studied (cf. \cite{BF,ferris,SW,CZ,zheng nonlinear 2014,ZY}).  In the case that $\varphi(t)=ct^2$, (LWP) means that $\bar x$ is a second-order local minimizer of $f$.
When $f$ undergoes tilt perturbations, under the name of ``uniform second-order growth condition", Bonnans and Shapiro \cite{BS} essentially introduced the following notion:  {\it $\bar x$ is said to be  a stable second order local  minimizer of $f$ if there exist $\kappa\in(0,\;+\infty)$ and neighborhoods $U^*$ of 0 and $U$ of $\bar x$ such that for every $u^*\in U^*$ there exists $x_{u^*}\in U$, with $x_0=\bar x$, satisfying
\begin{equation}\label{1.3}
\kappa\|x-x_{u^*}\|^2\leq f_{u^*}(x)-f_{u^*}(x_{u^*})\quad\forall x\in U,
\end{equation}
where $f_{u^*}:=f-u^*$.} In an earlier paper than \cite{BS}, Poliquin and Rockafellar \cite{poliquin and rockafellar} first introduced and studied another kind of   stability with respect to tilt perturbations: {\it $f$ is said to give a tilt-stable local minimum at $\bar x$ if $f(\bar x)$ is finite and there exist $\delta,r,L\in(0,\;+\infty)$ and $M:B_{X^*}(0,\delta)\rightarrow B_X(\bar x,r)$ with $M(0)=\bar x$ such that
\begin{equation}\label{1.1}
M(u^*)\in\mathop{\arg\min}_{x\in B_X(\bar x,r)}f_{u^*}(x)
\end{equation}
and}
\begin{equation}\label{1.2}
\|M(u_1^*)-M(u_2^*)\|\leq L\|u_1^*-u_2^*\|\quad\forall u_1^*,u_2^*\in B_{X^*}(0,\delta).
\end{equation}
In this paper, using admissible functions, we introduce and  study the following more general stability with respect to tilt perturbations.
\begin{definition}
{\it Given two admissible functions $\varphi,\psi:\mathbb{R}_+\rightarrow\mathbb{R}_+$ and a proper lower semicontinuous function $f$ on a Banach space $X$, we say that}\\
(i) {\it $f$ has  stable local well-posedness at $\bar x\in{\rm dom}(f)$ with respect to $\varphi$ (in brief, $\varphi$-SLWP) if
there exist $\delta,r,\tau,\kappa\in(0,\;+\infty)$ such that for every $u^*\in B_{X^*}(0,\delta)$ there exists $x_{u^*}\in B_X[\bar x,r]$, with $x_0=\bar x$, satisfying
\begin{equation}\label{1.5}
\varphi(\kappa\|x-x_{u^*}\|)\leq\tau(f_{u^*}(x)-f_{u^*}(x_{u^*}))\quad\forall x\in B_X[\bar x,r],
\end{equation}
where $B_X[\bar x,r]$ denote the closed ball of $X$ with center $\bar x$ and radius $r$.}\\
(ii) {\it $f$ is said to have a $\psi$-tilt-stable local minimum  at $\bar x$ (in brief, $\psi$-TSLM) if  there exist $\delta,r,\kappa,\tau\in(0,\;+\infty)$ and $M:B_{X^*}(0,\delta)\rightarrow B_X[\bar x,r]$ with $M(0)=\bar x$ such that (\ref{1.1}) holds and}
\begin{equation}\label{1.4}
\kappa\|M(u_1^*)-M(u_2^*)\|\leq\psi(\tau\|u_1^*-u_2^*\|)\quad\forall u_1^*,u_2^*\in B_{X^*}(0,\delta).
\end{equation}
\end{definition}
In the special case when $\varphi(t)= t^2$ and $\psi(t)=t$, the corresponding $\varphi$-{\it SLWP} and $\psi$-{\it TSLM} reduce to the stable second order local minimizer and tilt-stable local minimum, respectively. Many authors have studied the tilt-stable local minimum and  stable second order local minimizer. In 1998, Poliquin and Rockafellar \cite{poliquin and rockafellar} proved that if a proper lower semicontinuous function $f$ on $\mathbb{R}^n$ is prox-regular and subdifferentially continuous at $(\bar x,0)$ then $f$ gives a tilt stable minimum at $\bar x$ if and only if the second subdifferential $\partial^2f(\bar x,0)$ is positively definite. In 2008, under the convexity assumption of $f$, Arag\'{o}n Artacho and  Geoffroy \cite{aragon 2008} first studied the stable second order local minimizer of $f$ in terms of  the subdifferential mapping $\partial f$ and proved that $\bar x\in{\rm dom}(f)$ is a stable second order local  minimizer of $f$ if and only if  $\partial f$ is strongly metrically regular at $(\bar x, 0)$. In 2013, under the finite dimension assumption, Drusvyatskiy and Lewis \cite{DL}  extended Arag\'{o}n Artacho and  Geoffroy's result to the prox-regularity and subdifferential continuity case. Recently, these works have been pushed  by Drusvyatskiy,  Mordukhovich,  Nghia and  Outrata (cf.\cite{DMN, MN1, MN2, MN3, MO}). Zheng and Ng \cite{zheng siam 2014} further considered the H\"{o}lder tilt stability and the stable H\"{o}lder  local minimizer. This paper will consider the corresponding issues for $\psi$-{\it TSLM} and $\varphi$-{\it SLWP}.

To study $\varphi$-{\it SLWP} in terms of subdifferential mappings, we adopt  the following extension of the metric regularity.
\begin{definition}
Let $\psi:\mathbb{R}_+\rightarrow\mathbb{R}_+$ be an admissible  function and let $F$ be a multifunction between Banach spaces $X$ and $Y$ with $(\bar x,\bar y)\in{\rm gph}(F):=\{(x,y)\in  X\times Y:\;y\in F(x)\}$.\\
(i) $F$ is said to be metrically $\psi$-regular  at $(\bar x,\bar y)$ if there exist $r,\tau,\kappa\in(0,\;+\infty)$ such that
\begin{equation}\label{1.6}
\psi(\tau d(x,F^{-1}(y)))\leq \kappa d(y,F(x))\quad\forall (x,y)\in B_X(\bar x,r)\times B_Y(\bar y,r).
\end{equation}
(ii) $F$ is said to be strongly metrically $\psi$-regular at $(\bar x,\bar y)$ with respect to $\psi$ if  there exist $r,\tau,\kappa,\delta\in(0,\;+\infty)$ such that (\ref{1.6}) holds and $F^{-1}(y)\cap B_X(\bar x,\delta)$ is a singleton for all $y\in B_Y(\bar y,r)$.
\end{definition}\\
In the case when $\psi(t)=t$, the metric $\psi$-regularity is just the  metric regularity, which is a fundamental notion in variational analysis and well studied (cf. \cite{BS,DR,I,M,S,Zn3,ZN} and the references therein). When $\psi(t)=t^p$ with $p\in(0,\;+\infty)$,
(\ref{1.6}) means the so-called H\"{o}lder metric regularity of $F$ at $(\bar x,\bar y)$ (cf. \cite{FH,zheng siam 2014}). In Section 3,  we prove that  $f$  has  $\varphi$-{\it SLWP} at $\bar x$  if $\partial f$ is strongly metrically $\varphi'$-regular at $(\bar x,0)$ and that $\partial\overline{\rm co}(f+\delta_{B_X[\bar x,r]})$ is strongly metrically $\varphi'$-regular at $(\bar x,0)$ for some $r>0$ if $f$  has  $\varphi$-{\it SLWP} at $\bar x$; the later seems to be new even in the case when $\varphi(t)=t^2$.  In particular, under the convexity assumption on $f$, $f$ has  $\varphi$-{\it SLWP} at $\bar x$  if and only if $\partial f$ is strongly metrically $\varphi'$-regular at $(\bar x,0)$.

On one hand, given any two admissible functions $\varphi$ and $\psi$, we cannot expect that  $\varphi$-{\it SLWP} and $\psi$-{\it TSLM} are relevant. On the other hand, corresponding to the special case when $\varphi(t)=t^2$ and $\psi(t)=t$, Drusvyatskiy and Lewis \cite{DL} did prove that  the stable second order local minimizer and tilt-stable local minimum  are equivalent. Thus,  it is natural to ask whether there exists an exact relationship between $\varphi$ and $\psi$ such that  $\varphi$-{\it SLWP} and $\psi$-{\it TSLM} are equivalent. In Section 4, we  find that the equality $\psi(t)=(\varphi')^{-1}(t)$ is such a relationship. In particular, under some mild assumption and with the help of some techniques used in \cite{DL, MN1, zheng siam 2014, zheng nonlinear 2014}, we prove  that a proper lower semicontinuous function $f$ on a Banach space has $\varphi$-{\it SLWP} at $\bar x$ if and only if $f$ has  $(\varphi')^{-1}$-{\it TSLM} at $\bar x$.

Note that every  small linear perturbation $f_{u^*}$ of $f$ has an isolated minimizer around $\bar x$ if  $f$ has $\varphi$-{\it SLWP}  at $\bar x$.  In Section 5, we consider the stable weak well-posedness for the non-isolated minimizer case and obtain some interesting results.

In Section 6, in terms of  `generalized positive definiteness' of the second subdifferential $\partial^2f$, we provide a sufficient
condition for the subdifferential mapping $\partial f$ to be metrically regular with respect to an admissible function, which results in a sufficient condition for $f$ to have stable well-posedness in the convexity setting.

\section{Preliminaries}
\setcounter{section}{2}  \setcounter{equation}{0}
Let $X$ be a Banach space with the topological dual $X^*$. We denote by $B_X$ and $B_{X^*}$ the closed unit balls of $X$ and $X^*$, respectively.  For a proper lower semicontinuous function $f:X\rightarrow \mathbb{R}\cup\{+\infty\}$, we denote by ${\rm dom}(f)$ the effective domain  of $f$, that is,
$$
\textrm{dom}(f):=\{x\in X:f(x)<+\infty\} .
$$
For $x\in {\rm dom}(f)$ and $h\in X$, let $f^\uparrow(x,h)$ denote the generalized directional derivative introduced  by Rockafellar (cf. \cite{clarke}); that is,
$$
f^\uparrow(x,h):= \lim\limits_{\varepsilon\downarrow 0}\limsup\limits_{u\stackrel{f}{\rightarrow}x,t\downarrow 0}
\inf\limits_{w\in h+\varepsilon B_X}\frac{f(u+tw)-f(u)}{t},
$$
where the expression $u\stackrel{f}{\rightarrow}x$ means that $u\rightarrow x\;{\rm and}\;f(u)\rightarrow f(x)$.
Let $\partial f(x)$ denote the Clarke-Rockafellar subdifferential of $f$ at $x$, that is,
$$
\partial f(x):=\{x^*\in X^*:\;\langle x^*,h\rangle \leq f^\uparrow(x,h)\;\;\;\forall h\in X\}.
$$
In the case when $f$ is locally Lipschitzian around $x$, $f^\uparrow(x,h)$ reduces to the Clarke directional derivative
$$
f^\circ(x,h):= \limsup\limits_{u\rightarrow x,t\downarrow 0}\frac{f(u+th)-f(u)}{t}\quad \forall h\in X.
$$
It is well known that if  $f$ is convex, then
$$
\partial f(x)=\{x^*\in X^*:\;\langle x^*,y-x\rangle\leq f(y)- f(x)\quad\forall y\in X\}\quad\forall x\in{\rm dom}(f).
$$

Recall that the conjugate function $f^*$ of $f$ is a weak$^*$-lower semicontinuous convex function  on $X^*$ such that
$$
f^*(u^*):=\sup\{\langle u^*,x\rangle-f(x):\;x\in X\}=-\inf\{f_{u^*}(x):\;x\in X\}\;\;\forall u^*\in X^*,
$$
where
\begin{equation}\label{2.1}
f_{u^*}(x):=f(x)-\langle u^*,x\rangle\quad\forall x\in X.
\end{equation}
It follows that ${\rm dom}(f^*)\not=\emptyset$ if and only  $f$ is bounded below by a continuous linear functional.
For $x^*\in X^*$ and $x\in X$, it is easy to verify  that
$$
f^*(x^*)=\langle x^*,x\rangle-f(x)\Longrightarrow x\in\partial f^*(x^*)
$$
In the case when $f$ is convex, it is well known (cf. \cite[P.88]{L}) that
$$
f=f^{**}\;\;{\rm and}\;\;x^*\in\partial f(x)\Longleftrightarrow x\in\partial f^*(x^*).
$$

Let $g:\mathbb{R}_+\rightarrow\mathbb{R}_+$ be a convex function. Then  the directional derivative
$$g_+'(t):=\lim\limits_{s\rightarrow0^+}\frac{g(t+s)-g(t)}{s}$$
always exists for all $t\in\mathbb{R}_+$, and $g_+'$ is nondecreasing on $\mathbb{R}_+$. It is known (cf. \cite[Theeorem 2.1.5]{CZ}) that $g_+'$ is increasing on $\mathbb{R}_+$ if and only if $g$ is strictly convex, namely
$$g(\lambda t_1+(1-\lambda)t_2)<\lambda g(t_1)+(1-\lambda)g(t_2)$$
for any $\lambda\in(0,\;1)$ and $t_1,t_2\in\mathbb{R}_+$ with $t_1\not=t_2$.
It is also known that the convex function $g$ is differentiable on $\mathbb{R}_+$ if and only if $g_+'$ is continuous on $\mathbb{R}_+$.

Recall that an admissible function $\varphi$ is a nondecreasing function $\varphi:\mathbb{R}_+\rightarrow\mathbb{R}_+$ such that $\varphi(0)=0$ and
\begin{equation}\label{2.1'}
\varphi(t)\rightarrow0\Rightarrow t\rightarrow0.
\end{equation}
If the admissible function $\varphi$ is convex, it is  easy from \cite[Theorem 2.1.5]{CZ} to verify that
\begin{equation}\label{3.1}
0<\varphi_+'(t_1)\leq\varphi_+'(t_2)\;\;{\rm for\;all}\;t_1,t_2\in(0,\;+\infty)\;{\rm with}\;t_1\leq t_2.
\end{equation}
For convenience, for each $\alpha\in(0,\;1)$, let
\begin{equation}\label{2.3}
\varphi_\alpha'(t):=\frac{1}{\alpha}\varphi_+'\left(\frac{t}{1-\alpha}\right)\quad\forall t\in\mathbb{R}_+.
\end{equation}

\section{Stable well-posedness}
In this section, let $\varphi:\mathbb{R}_+\rightarrow\mathbb{R}_+$ be a convex admissible function. The following lemma, established  in   \cite{yao and zheng 2014}, is very useful in the proof of the main result in this section.

\begin{lemma}
Let $g$ be a proper lower semicontinuous function on a Banach space $X$. Let $\bar x\in \mathop{\arg\min}_{z\in X} g(z)$, $\alpha\in(0,\;1)$ and $\delta\in(0,\;+\infty)$ be such that
$$
\varphi'_\alpha\big(d(x,\mathop{\arg\min}_{z\in X} g(z))\big)\leq d(0,\partial g(x))\quad\forall x\in B_X(\bar x,\delta)\setminus \mathop{\arg\min}_{z\in X} g(z)
$$
where $\varphi_\alpha'$ is as in (\ref{2.3}).
Then,
$$
\varphi(d(x,\mathop{\arg\min}_{z\in X} g(z)))\leq g(x)-g(\bar x)\quad\forall x\in B_X(\bar x,\frac{\delta}{1+\alpha}).
$$
\end{lemma}

Let $g:X\rightarrow\mathbb{R}\cup\{+\infty\}$ be a proper lower semicontinuous function, $u\in{\rm dom}(g)$ and $\beta>0$.
For convenience, we adopt the following notation:
\begin{equation}\label{a3.1}
M_g(u,\beta):=\mathop{\arg\min}_{z\in B[u,\beta]} g(z).
\end{equation}
Applying Lemma 3.1 to  $g=f+\delta_{B_X[\bar x,r]}$ and $\varphi(t)=\frac{\alpha}{\tau\kappa(1-\alpha)}\psi(\tau (1-\alpha) t)$, we have the following lemma.
\begin{lemma}
Let $\psi$ be a convex admissible function and let $f$ be  a proper lower semicontinuous function on a Banach space $X$. Let $\bar x\in {\rm dom}(f)$ and $r>0$ be such that
$$f(\bar x)=\min\limits_{x\in B_X[\bar x,r]}f(x).$$
Suppose that there exist  $\tau,\kappa,\delta\in(0,\;+\infty)$ such that
$$
\psi'_+\big(\tau d(x,M_f(\bar x,r))\big)\leq \kappa d(0,\partial f(x))\quad\forall x\in B_X(\bar x,\delta)\setminus M_f(\bar x,r).
$$
Then, for any $\alpha\in(0,\;1)$,
$$
\psi\big(\tau(1-\alpha)d(x,M_f(\bar x,r))\big)\leq \frac{\tau\kappa(1-\alpha)}{\alpha}(f(x)-f(\bar x))\quad\forall x\in B_X(\bar x,\frac{\min\{\delta,r\}}{1+\alpha}).
$$
\end{lemma}

With the help of Lemma 3.2, we can prove the following sufficient condition for the stable well-posedness.

\begin{theorem}
\label{strong well}
Let $f$ be a proper lower semicontinuous function  on a Banach space $X$ and let  $\bar x\in {\rm dom} (f)$ be a local minimizer of $f$.  Suppose that $\partial f$ is strongly metrically $\varphi_+'$-regular at $(\bar x,0)$. Then $f$ has $\varphi$-{\it SLWP} at $\bar x$.
\end{theorem}

\begin{proof}
By the assumption, there exist $r,\gamma,\delta,\tau,\kappa\in(0,\;+\infty)$ such that
\begin{equation}
\min_{x\in B_X[\bar x, r]}f(x)=f(\bar x), \label{theo1}
\end{equation}
\begin{equation}\label{3.10'}
(\partial f)^{-1}(u^*)\cap B_X(\bar x,\gamma)=\{x_{u^*}\}\quad\forall u^*\in B_{X^*}(0,\delta)
\end{equation}
and
\begin{equation}
\;\;\varphi_+'(\tau d(x,(\partial f)^{-1}(u^*))) \leq \kappa d(u^*,\partial f(x))\quad \forall (x,u^*)\in B_X(\bar x, \delta)\times B_{X^*}(0,\delta). \label{theo2}
\end{equation}
Hence $x_0=\bar x$ and
$$
d(\bar x,(\partial  f)^{-1}(u^*))=d(\bar x,(\partial f)^{-1}(u^*)\cap B_X(\bar x,\gamma))=\|\bar x-x_{u^*}\|\quad\forall u^*\in B_{X^*}(0,\delta).
$$
Setting $x=\bar x$ in inequality (\ref{theo2}) and  noting that $0\in \partial f(\bar x)$, it follows that
$$
\varphi_+'(\tau\|\bar x-x_{u^*}\|) \leq \kappa d(u^*,\partial f(\bar x))\leq \kappa\|u^*\| \quad \forall u^*\in B_{X^*}(0,\delta).\label{theo7}
$$
Let
$$\eta:=\frac{\min\{\delta,r,\gamma\}}{16}\;\;{\rm and}\;\; \delta_1:=\min\left\{\delta,,r,\gamma,\frac{\varphi'_+(\tau\eta)}{\kappa},\frac{2\varphi(2\tau\eta)}{9\tau\kappa\eta}\right\}.$$
Then,
$$
\varphi_+'(\tau \|\bar x-x_{u^*}\|) \leq \kappa\|u^*\|< \varphi_+'(\tau\eta)\quad \forall u^*\in B_{X^*}(0,\delta_1)
$$
and so
\begin{equation}
\|x_{u^*}-\bar x\| <\eta\leq\frac{\gamma}{16}\quad \forall \ u^*\in B_{X^*}(0,\delta_1) \label{theo8}
\end{equation}
(because $\varphi_+'$ is nondecreasing). This and  the definition of $\eta$ imply that
\begin{equation}
B_X(x_{u^*}, 6\eta) \subset B_X(\bar x, 7\eta)\subset B_X(\bar x,\gamma)\cap B_X(\bar x,\delta)\quad\forall u^*\in B_{X^*}(0,\delta_1). \label{theo9}
\end{equation}
Thus, by (\ref{3.10'}), one has
$$(\partial f)^{-1}(u^*)\cap B(x_{u^*},6\eta)=\{x_{u^*}\}\quad\forall u^*\in B_{X^*}(0,\delta_1),$$
and so
$$d(x,(\partial f)^{-1}(u^*))=\|x-x_{u^*}\|\quad\forall u^*\in B_{X^*}(0, \delta_1)\;{\rm and}\;x\in B_X(x_{u^*}, 3\eta).$$
Noting that $\partial f_{u^*}(x)=\partial f(x)-u^*$, it follows from  (\ref{theo2}) and  (\ref{theo9}) that
\begin{equation}
\quad\varphi_+'(\tau\|x-x_{u^*}\|) \leq\kappa d(0,\partial f_{u^*}(x))\quad\forall u^*\in B_{X^*}(0, \delta_1)\;{\rm and}\;x\in B_X(x_{u^*}, 3\eta).\label{theo6}
\end{equation}
We claim that
\begin{equation}
M_{f_{u^*}}(x_{u^*},3\eta)=\{x_{u^*}\}\quad \forall u^*\in B_{X^*}(0,\delta_1),\label{theo11}
\end{equation}
where $M_{f_{u^*}}(x_{u^*},3\eta)$ is defined as in (\ref{a3.1}). Granting this, by (\ref{theo8}) and   (\ref{theo6}), we have that
\begin{equation}\label{3.2'na}
M_{f_{u^*}}(\bar x,\eta)=\{x_{u^*}\}\quad\forall u^*\in B_{X^*}(0,\delta_1)
\end{equation}
and
$$\varphi_+'\left(\tau d\big(x,M_{f_{u^*}}(x_{u^*},3\eta)\big)\right)\leq \kappa d(0,\partial f_{u^*}(x))$$
for all $u^*\in B_{X^*}(0, \delta_1)$ and $x\in B_X(x_{u^*}, 3\eta)$. Thus, by Lemma 3.2 (applied to $f=f_{u^*}$, $\bar x=x_{u^*}$, $r=3\eta$, $\delta=\delta_1$ and $\alpha=\frac{1}{2}$), one has
$$
\varphi \left(\frac{\tau}{2} \|x-x_{u^*}\|\right)=\varphi\left(\frac{\tau}{2}d\big(x,M_{f_{u^*}}(x_{u^*},3\eta)\big)\right)\leq \tau\kappa(f_{u^*}(x)-f_{u^*}(x_{u^*}))
$$
for all  $u^*\in B_{X^*}(0, \delta_1)$ and $x\in B_X(x_{u^*}, 2\eta)$.
Noting (by (\ref{theo8})) that
$x_{u^*}\in B_X[\bar x, \eta] \subset B_X(x_{u^*}, 2\eta),$
it follows that
$$\varphi (\frac{\tau}{2} \|x-x_{u^*}\|)\leq  \tau\kappa(f_{u^*}(x)-f_{u^*}(x_{u^*}))\quad\forall (x,u^*)\in B_X[\bar x,\eta]\times B_{X^*}(0,\delta_1).$$
This and (\ref{3.2'na}) imply that $f$ has $\varphi$-{\it SLWP} at $\bar x$.
It remains to show that (\ref{theo11}) holds. By (\ref{theo8}), one has   $B_X[x_{u^*}, 3\eta] \subset B_X(\bar x, 4\eta)$ for all $u^*\in B_{X^*}(0,\delta_1)$. Thus, to prove (\ref{theo11}),   we only need to show that
\begin{equation}
\{x_{u^*}\}=M_{f_{u^*}}(\bar x,4\eta)\quad \forall u^*\in B_{X^*}(0,\delta_1).\label{theo15}
\end{equation}
To do this, given a $u^*\in B_{X^*}(0,\delta_1)$,  for each $n\in \mathbb{N}$ take  $x_n\in B_X[\bar x, 4\eta]$  such that
\begin{equation}
f_{u^*}(x_n)<\inf_{x\in B_X[\bar x,  4\eta]}f_{u^*}(x)+n^{-2}.\label{theo16}
\end{equation}
It suffices to show that  $\|x_n-x_{u^*}\|\rightarrow0$. By Ekeland's variational principle and (\ref{theo16}), there exists $u_n\in B_X[\bar x,  4\eta]$ such that
\begin{equation}
\|u_n-x_n\| \leq n^{-1}\label{theo17}
\end{equation}
and
\begin{equation}
f_{u^*}(u_n) \leq f_{u^*}(x)+n^{-1}\|x-u_n\|\quad \forall x\in B_X[\bar x,  4\eta].\label{theo18}
\end{equation}
Thus, $\|u_n-\bar x\|\leq4\eta<\frac{\gamma}{3}$. It follows from (\ref{3.10'}) and (\ref{theo8}) that
\begin{equation}\label{3.11'na}
d(u_n,(\partial f)^{-1}(u^*))=\|u_n-x_{u^*}\|.
\end{equation}
We claim that $u_n$ lies in the open ball $B_X(\bar x, 4\eta)$ for all sufficiently large $n\in \mathbb{N}$. Granting this,  (\ref{theo18}) implies that
$$
0\in \partial f_{u^*}(u_n)+n^{-1}B_{X^*}
$$
for all sufficiently large $n$.
Since $\|u_n-\bar x\|\leq4\eta<\delta$, it follows from (\ref{theo2}) and (\ref{3.11'na}) that
$$
\varphi_+'(\tau\|u_n-x_{u^*}\|)\leq \kappa d(u^*,\partial f(u_n)) =\kappa d(0,\partial f_{u^*}(u_n))\leq\kappa n^{-1}
$$
for all sufficiently large $n$.
Thus, by  (\ref{3.1}), one has  $\|u_n-x_{u^*}\|\rightarrow 0$. This, together with (\ref{theo17}), shows that $\|x_n- x_{u^*}\|\rightarrow0$. Finally we prove that $u_n$ lies in the open ball $B_X(\bar x, 4\eta)$ for all sufficiently large $n\in \mathbb{N}$. Setting $u^*=0$ in (\ref{theo2}), one has
$$\varphi_+'(\tau d(x,(\partial f)^{-1}(0))) \leq\kappa d(0,\partial f(x))\quad \forall x\in B_X(\bar x, \delta).
$$
Letting $\delta_0:=\frac{1}{2}\min\{\delta,r,\gamma\}$ and noting (by (\ref{theo1}) and (\ref{3.10'}))  that
$$\{\bar x\}=(\partial f)^{-1}(0)\cap B_X(\bar x,2\delta_0)=M_f(\bar x,\delta_0),$$
it follows that
$$
\varphi_+'(\tau d(x,M_f(\bar x,\delta_0))) \leq \kappa d(0,\partial f(x))\quad \forall x\in B_X(\bar x, \delta_0).
$$
Thus, by Lemma 3.2 (applied to $\alpha=\frac{1}{2}$),
\begin{equation}
\varphi (\frac{\tau}{2}\|x-\bar x\|)\leq\tau\kappa(f(x)-f(\bar x))\quad \forall x \in B_X(\bar x, \frac{2\delta_0}{3}).\label{theo13}
\end{equation}
By the  definition of $\eta$, one has  $u_n\in B_X[\bar x, 4\eta] \subset B_X(\bar x, \frac{2\delta_0}{3})$. Given a $u^*\in B_{X^*}(0,\delta_1)$, it follows from (\ref{theo13}), (\ref{theo18}) and the choice of $\delta_1$ that
\begin{eqnarray*}
\varphi (\frac{\tau}{2}\|u_n-\bar x\|) &\leq&\tau\kappa(f(u_n)-f(\bar x))\\
&=&\tau\kappa(f_{u^*}(u_n)-f_{u^*}(\bar x)+\langle u^*,u_n-\bar x\rangle)\\
& \leq&\tau\kappa(\frac{1}{n}\|u_n-\bar x\|+\|u^*\|\|u_n-\bar x\|)\\
&\leq&\tau\kappa(n^{-1}+\delta_1)\|u_n-\bar x\|\\
&\leq& \tau\kappa\left(\frac{1}{n}+\frac{2\varphi(2\tau\eta)}{9\tau\kappa\eta}\right)4\eta
\end{eqnarray*}
and so
$$\limsup_{n\rightarrow +\infty} \varphi(\frac{\tau}{2}\|u_n-\bar x\|)\leq\frac{8\varphi(2\tau\eta)}{9} < \varphi\left(2\tau\eta\right).$$
Noting that $\varphi$ is nondecreasing, it follows that $\|u_n-\bar x\| <4\eta$ for all sufficiently large $n$.   The proof is complete.
\end{proof}

Even in the special case when $\varphi(t)=t^2$, the converse of Theorem 3.3 is not necessarily true (see \cite[Example 3.4]{DL}).
This and Theorem 3.3 make the following necessity result meaningful.

Let $g$ be a proper lower semicontinuous function on a Banach space $X$ such that $-\infty<\inf\limits_{x\in X}g(x)$, and let $\overline{\rm co}g$ denote the convex envelope of $g$, that is,
${\rm epi}(\overline{\rm co}g)=\overline{\rm co}({\rm epi}(g))$. Then, $\overline{\rm co}g$ is a proper lower semicontinuous convex function,
$$g^{**}=\overline{\rm co}g\;\;{\rm and}\;\;g^*=(\overline{\rm co}g)^*$$
where $g^*$ and $g^{**}$ denote respectively the conjugate function and twice conjugate function of $g$ (cf. \cite[Theorem 2.3.1]{CZ} and \cite[Theorem 2.3.4]{CZ}).

\begin{theorem}
Let $\varphi$ be a strictly convex differentiable admissible function and $f$ be a proper lower semicontinuous function  on a Banach space $X$. Suppose that $f$ has $\varphi$-{\it SLWP} at  $\bar x\in {\rm dom} (f)$. Then there exists $r>0$ such that $\partial\overline{\rm co}(f+\delta_{B_X[\bar x,r]})$ is strongly metrically $\varphi'$-regular at $(\bar x,0)$.
\end{theorem}

We postpone the proof of Theorem 3.4 at the end of Section 4.
The following corollary is immediate from Theorems 3.3 and 3.4.
\begin{corollary}
Let $\varphi$ be a strictly convex differentiable admissible function and $f$ be a proper lower semicontinuous convex function  on a Banach space $X$. Then $f$ has $\varphi$-{\it SLWP}  at  $\bar x\in {\rm dom} (f)$ if and only if $\partial f$ is strongly metrically $\varphi_+'$-regular at $(\bar x,0)$.
\end{corollary}

In the case when $\varphi(t)=t^2$,  Corollary 3.5 was established  by Arag\'{o}n Artacho and  Geoffroy \cite{aragon 2008}. In the  Asplund space case,  Mordukhovich and Nghia \cite{MN1} proved that $\partial f$ is strongly metrically regular at $(\bar x,0)$ if and only if there exist a neighborhood $U^*$ of 0, a neighborhood $U$ of $\bar x$ and a single-valued funciton $\vartheta:U^*\rightarrow U$ such that ${\rm gph}\vartheta={\rm gph}(\partial f)^{-1}\cap (U^*\times U)$ and
$$\tau\|x-u\|^2\leq f_{u^*}(x)-f_{u^*}(u)\quad\forall x\in U\;{\rm and}\;(u^*,u)\in{\rm gph}(\partial f)^{-1}\cap (U^*\times U),$$
where $\tau$ is a positive constant.

We conclude the section with a necessary condition for $\varphi$-{\it SLWP}, which is related to the following well-known optimality condition:
$$f(\bar x)=\min\limits_{x\in B_X(\bar x,r)}f(x)\Longrightarrow 0\in\partial f(\bar x).$$

\begin{proposition}
Let $\varphi$ be an admissible function and let $f$ be a proper lower semicontinuous function on a Banach space $X$.  Suppose that $f$ has $\varphi$-{\it SLWP} at $\bar x\in{\rm dom}(f)$. Then,
 $$0\in{\rm int}(\partial f(B_X(\bar x,\varepsilon)))\quad\forall \varepsilon\in(0,\;+\infty).$$
\end{proposition}

\begin{proof}
Since $f$ has $\varphi$-{\it SLWP} at $\bar x$, there exist $r,\delta,\tau,\kappa\in(0,\;+\infty)$ such that for every $u^*\in B_{X^*}(0,\delta)$ there exists $x_{u^*}\in B_X[\bar x,r]$, with $x_0=\bar x$, satisfying (\ref{1.5}). Hence
\begin{equation}\label{3.a}
x_{u^*}\in\mathop{\arg\min}_{z\in B_X[\bar x,r]}f_{u^*}(z)\quad\forall u^*\in B_{X^*}(0,\delta)
\end{equation}
and
\begin{eqnarray*}
\varphi(\kappa\|\bar x-x_{u^*}\|)&\leq& \tau(f_{u^*}(\bar x)-f_{u^*}(x_{u^*}))\\
&=&\tau(f(\bar x)-f(x_{u^*})-\langle u^*,\bar x-x_{u^*}\rangle)\\
&=&\tau(\min\limits_{z\in B_X[\bar x,r]}f(z)-f(x_{u^*})-\langle u^*,\bar x-x_{u*}\rangle)\\
&\leq&-\tau\langle u^*,\bar x-x_{u^*}\rangle\leq\tau r\|u^*\|
\end{eqnarray*}
for all $u^*\in B_{X^*}(0,\delta)$. Hence $\lim\limits_{u^*\rightarrow0}\varphi(\kappa\|\bar x-x_{u^*}\|)=0$. This and (\ref{2.1'}) imply that
\begin{equation}\label{3.b}
\lim\limits_{u^*\rightarrow0}\kappa\|\bar x-x_{u^*}\|=0.
\end{equation}
Thus, for any $\varepsilon>0$ there exists $\gamma\in(0,\;\delta)$ such that $\|\bar x-x_{u^*}\|<\min\{\varepsilon,r\}$ for all $u^*\in B_{X^*}(0,\gamma)$. It follows from  (\ref{3.a}) that $x_{u^*}$ is a local minimizer of $f_{u^*}$ for each $u^*\in B_{X^*}(0,\gamma)$. Hence
$$0\in\partial f_{u^*}(x_{u^*})=\partial f(x_{u^*})-u^*\subset\partial f(B_X(\bar x,\varepsilon))-u^*\quad\forall u^*\in B(0,\gamma),$$
which implies $B_{X^*}(0,\gamma)\subset \partial f(B_X(\bar x,\varepsilon))$. The proof is complete.
\end{proof}

{\bf Remark.} From (\ref{3.b}), $x_{u^*}$ in Definition 1.1(i) can be taken in the open ball $B_X(\bar x,r)$ (taking a smaller $\delta$ if necessary). Thus, from the concerned definitions, it is clear that $\bar x$ is a stable second order local minimizer of $f$ (i.e. uniform second order growth condition) if and only if $f$ has $\varphi$-{\it SLWP} at $\bar x$ with $\varphi(t)=t^2$.

\section{Tilt-stability with respect to an admissible function}
\setcounter{equation}{0}
In this section, we will provide some necessary conditions and characterizations for the  tilt-stable minimum with respect to an admissible function.  First, we provide two  lemmas which play important roles in the proofs of the main results in this section.  For a  continuous function $\omega:\mathbb{R}_+\rightarrow\mathbb{R}_+$ with $\omega(0)=0$,  recall (cf. \cite{journai and thibault}) that a proper lower semicontinuous extended real-valued function $g$ on a Banach space $E$ is $C^{1,\omega}$ smooth on $D\subset{\rm dom}(g)$ if  $g$ is Fr\'{e}chet differentiable on $D$ and
$$
\|\triangledown g(x_1)-\triangledown g(x_2)\|\leq \omega(\|x_1-x_2\|)\quad \forall x_1, x_2\in D.
$$

\begin{lemma}
\label{proposition2}
Let $\omega:\mathbb{R}_+\rightarrow\mathbb{R}_+$ be an increasing continuous function with $\omega^{-1}(0)=\{0\}$, $E$  be a Banach space and let $g:E\rightarrow\mathbb{R}\cup\{+\infty\}$ be a proper lower semicontinuous function. Let $\bar u\in E$ and $\delta>0$ be such that $g$ is $C^{1,\omega}$ smooth  on $B_E\big(\bar u,\delta+\omega^{-1}(2\omega(\delta))\big)\subset {\rm dom}(g)$.
Then
\begin{equation}
g^*(x^*)\geq \langle x^*,u\rangle-g(u)+ \int_0^{\|x^*-\triangledown g(u)\|} \omega^{-1}(s)ds \label{prop3}
\end{equation}
for all $(u, x^*)\in B_E(\bar u, \delta)\times B_{E^*}(\triangledown g(\bar u),\omega (\delta))$.
\end{lemma}

\begin{proof}
Let $\delta_0:=\delta+\omega^{-1}(2\omega(\delta))$.
Then
$$\|\triangledown g(x_1)-\triangledown g(x_2)\|\leq \omega(\|x_1-x_2\|)\quad \forall x_1, x_2\in B_E(\bar u,\delta_0)
$$
(because $g$ is  $C^{1,\omega}$ smooth on $B_E(\bar u,\delta_0)$). Hence,
\begin{eqnarray*}
g(v)-g(u)-\langle \triangledown g(u), v-u\rangle&=&\int_0^1\langle \triangledown g(u+t(v-u))-\triangledown g(u), v-u\rangle dt\\
&\leq& \int_0^1\omega (t\|v-u\|)\|v-u\|dt\\
&=&\int_0^{\|v-u\|} \omega (t)dt
\end{eqnarray*}
for all  $u,v\in B_E(\bar u, \delta_0)$.  Let $(u,x^*)\in B_E(\bar u,\delta)\times B_{E^*}(\triangledown g(\bar u), \omega (\delta))$. Then,
\begin{eqnarray*}
g^*(x^*)&\geq& \sup_{v\in B_E(\bar u, \delta_0)}\left\{\langle x^*,v\rangle-g(v)\right \}\\
&\geq& \sup_{v\in B_E(\bar u, \delta_0)}\left\{\langle x^*,v\rangle-g(u)-\langle \triangledown g(u), v-u\rangle-\int_0^{\|v-u\|} \omega (t)dt\right \}\\
&=&\langle x^*,u\rangle-g(u)+\sup_{v\in B_E(\bar u, \delta_0)}\left\{\langle x^*-\triangledown g(u),v-u\rangle-\int_0^{\|v-u\|} \omega (t)dt\right \}.\\
\end{eqnarray*}
Thus, to prove (\ref{prop3}), it suffices to show that
\begin{eqnarray}
\beta:&=&\sup_{v\in B_E(\bar u, \delta_0)}\left\{\langle x^*-\triangledown g(u),v-u\rangle-\int_0^{\|v-u\|} \omega (t)dt\right \}\nonumber\\
&\geq& \int_0^{\|x^*-\triangledown g(u)\|} \omega^{-1} (t)dt.\label{prop4}
\end{eqnarray}
To do this, take a sequence $\{z_n\}$ in $E$ such that each $\|z_n\|=1$ and
\begin{equation}
\langle x^*-\triangledown g(u),z_n\rangle \rightarrow \|x^*-\triangledown g(u)\|.\label{prop6}
\end{equation}
For each $n\in \mathbb N$, let
$$
v_n:=u+\omega^{-1}(\|x^*-\triangledown g(u)\|)z_n.
$$
Then
\begin{eqnarray*}
\|v_n-\bar u\| &\leq& \|u-\bar u\|+\omega^{-1}(\|x^*-\triangledown g(u)\|)\\
&<& \delta +\omega^{-1}(\|x^*-\triangledown g(\bar u)\|+\|\triangledown g(\bar u)-\triangledown g(u)\|)\\
&\leq& \delta + \omega^{-1}\left(\omega(\delta)+\omega(\|\bar u-u\|)\right)\\
&\leq& \delta + \omega^{-1}(2\omega(\delta ))= \delta_0.
\end{eqnarray*}
This and the definition of $\beta$ imply  that
\begin{eqnarray*}
\beta &\geq& \langle x^*-\triangledown g(u),v_n-u\rangle-\int_0^{\|v_n-u\|} \omega (t)dt\\
&=&\omega^{-1}(\|x^*-\triangledown g(u)\|)\langle x^*-\triangledown g(u),z_n \rangle-\int_0^{\omega^{-1}(\|x^*-\triangledown g(u)\|)} \omega (t)dt.
\end{eqnarray*}
It follows from  (\ref{prop6})  that
\begin {eqnarray*}
\beta &\geq& \omega^{-1}(\|x^*-\triangledown g(u)\|)\cdot \|x^*-\triangledown g(u)\|-\int_0^{\omega^{-1}(\|x^*-\triangledown g(u)\|)} \omega (t)dt\\
&=&\int_0^{\omega^{-1}(\|x^*-\triangledown g(u)\|)} td\omega (t)\\
&=&\int_0^{\|x^*-\triangledown g(u)\|} \omega^{-1} (s)ds
\end {eqnarray*}
(the first equality holds because of integration by parts),
which verifies  (\ref{prop4}). The proof is complete.
\end{proof}\\

From \cite[Theorem 3.5.12]{CZ}, one has the following result: if $g$ is convex and $C^{1,\omega}$ smooth, then  there exists a convex admissible function $\omega_1$ such that
$$
g^*(x^*)\geq g^*(\triangledown g(u)) +\langle x^*-\triangledown g(u),u\rangle+\omega_1(\|x^*-\triangledown g(u)\|),$$
which implies
$$
g^*(x^*)\geq \langle x^*,u\rangle-g(u)+ \omega_1(\|x^*-\triangledown g(u)\|).
$$
In contrast, without the convexity assumption on $g$, Lemma~\ref{proposition2}  provides a  quantitative and calculable formula between  $g$ and $g^*$.

Let $Z$ be a Banach space and recall that a set-valued mapping $F:Z\rightrightarrows Z^*$ is lower semicontinuous at $z_0\in{\rm dom}(F):=\{z\in Z:\;F(z)\not=\emptyset\}$ if  for any open set $V$ with $V\cap F(z_0)\not=\emptyset$ there exists a neighborhood $U$ of $z_0$ such that $V\cap F(z)\not=\emptyset$ for all $z\in U$. Let $\omega:\mathbb{R}_+\rightarrow\mathbb{R}_+$ be such that
$$\lim\limits_{t\rightarrow0^+}\omega(t)=\omega(0)=0.$$
It is routine to verify that the lower semicontinuity of $F$ at $z_0$ is implied by the following $\omega$-Lipschitz continuity (L$_{\omega}$): there exists $\delta>0$ such that
$$F(z_1)\subset F(z_2)+\omega(\|z_1-z_2\|)B_{Z^*}\quad\forall z_1,z_2\in B_Z(z_0,\delta).\leqno{\rm (L_{\omega})}$$
For $(z_0,z_0^*)\in {\rm gph}(F):=\{(z,z^*):\;z\in Z\;{\rm and}\;z^*\in F(z)\}$, as an extension of the Aubin property, we consider the following property: there exists $\gamma>0$ such that
\begin{equation}\label{4.4}
F(z_1)\cap B_{Z^*}(z_0^*,\gamma)\subset F(z_2)+\omega(\|z_2-z_1\|)B_{Z^*}\quad\forall z_1,z_2\in B_Z(z_0,\delta).
\end{equation}
Clearly, (L$_{\omega}$) implies (\ref{4.4}), but the converse implication is not necessarily true. Indeed, (\ref{4.4}) does not necessarily imply the lower semicontinuity of $F$ at $z_0$. For example, let $Z=\mathbb{R}$ and $F(0)=\{0,2\}$ and $F(t)=\{\omega(|t|)\}:=\{|t|\}$ for all $t\in\mathbb{R}\setminus\{0\}$. Then, $F(z_1)\cap B_{\mathbb{R}}(0,1)=\{|z_1|\}$ and $F(z_2)+\omega(|z_1-z_2|)B_{\mathbb{R}}=|z_2|+|z_1-z_2|B_{\mathbb{R}}$ for all $z_1,z_2\in B_{\mathbb{R}}(0,1)$; hence
$$F(z_1)\cap B_{\mathbb{R}}(0,1)\subset F(z_2)+\omega(|z_1-z_2|)B_{\mathbb{R}}\quad\forall z_1,z_2\in B_{\mathbb{R}}(0,1).$$
On the other hand, since $B_{\mathbb{R}}(2,1)\cap F(0)=\{2\}$ and $B_{\mathbb{R}}(2,1)\cap F(z)=\emptyset$ for all $z\in B_{\mathbb{R}}(0,1)\setminus\{0\}$, $F$ is not semicontinuous at 0.

Recall that a set-valued mapping $F$ is monotone if
$$0\leq\langle z_1^*-z_2^*,z_1-z_2\rangle\quad\forall (z_1,z_1^*),(z_2,z_2^*)\in {\rm gph}(F).$$
Kenderov \cite{PK} proved the following interesting result on the single-valuedness of a monotone mapping.

{\bf Result K.} {\it Let $F$ be a monotone mapping from a Banach space $Z$ to $Z^*$ and suppose that $F$ is lower semicontinuous at $z_0$ with $F(z_0)\not=\emptyset$. Then, $F(z_0)$ is a singleton.}

Since (\ref{4.4}) does not imply the lower semicontinuity of $F$ at $z_0$, the following lemma provides a supplement of Result K.

\begin{lemma}
Let $\omega:\mathbb{R}_+\rightarrow\mathbb{R}_+$ be a function such that $\lim\limits_{t\rightarrow0^+}\omega(t)=\omega(0)=0$ and let $F$ be a monotone mapping from a Banach space $Z$ to $Z^*$ with $(z_0,z_0^*)\in{\rm gph}(F)$. Suppose that there exist $\gamma,\delta\in(0,\;+\infty)$ such that (\ref{4.4}) holds.
Let $\gamma':=\sup\{t\geq0:\;[0,\;t]\subset \omega^{-1}[0,\;\gamma)\}$ and $\delta':=\min\{\delta,\gamma'\}$. Then, $F(z)$ is a singleton for all $z\in B_Z(z_0,\delta')$.
\end{lemma}

\begin{proof}
Let $z\in B_Z(z_0,\delta')$. Then, $\|z-z_0\|<\delta'\leq\gamma'$ and so $\omega(\|z-z_0\|)<\gamma$. This and (\ref{4.4}) imply that
$z_0^*\in F(z)+\omega(\|z-z_0\|)B_{Z^*}$.
Hence  there exists $v_z^*\in F(z)$ such that $\|v_z^*-z_0^*\|\leq\omega(\|z-z_0\|)<\gamma$.
It suffices to show that $F(z)\setminus\{v^*_z\}=\emptyset$. To do this, suppose to the contrary that there exists $z^*\in F(z)$ such that $v_z^*\not=z^*$. Then, there exists $h\in Z$ with $\|h\|=1$ such that
\begin{equation}\label{4.5}
\langle v_z^*-z^*,h\rangle<0.
\end{equation}
Since $\|z-z_0\|<\delta$, there exists  a sequence $\{\varepsilon_n\}\subset (0,\;+\infty)$ converging to 0 such that $\{z+\varepsilon_n h\}\subset B_Z(z_0,\delta)$. It follows from (\ref{4.4}) that
$$v_z^*\in F(z)\cap B(z_0^*,\gamma)\subset F(z+\varepsilon_nh)+\omega(\varepsilon_n)B_{Z^*}\quad \forall n\in\mathbb{N}.$$
Hence,  for any $n\in\mathbb{N}$ there exists $z_n^*\in F(z+\varepsilon_n h)$ such that $\|z_n^*-v_z^*\|\leq\omega(\varepsilon_n)\rightarrow0$. On the other hand, by the monotonicity of $F$,
$$0\leq\langle z_n^*-z^*,\varepsilon_nh\rangle=\varepsilon_n\langle z_n^*-z^*,h\rangle=\varepsilon_n(\langle z_n^*-v_z^*,h\rangle+\langle v_z^*-z^*,h\rangle)\quad\forall n\in\mathbb{N}.$$
Therefore,
$$\langle v^*_z-z^*,h\rangle\geq-\langle z_n^*-v_z^*,h\rangle\geq-\|z_n^*-v_z^*\|\geq-\omega(\varepsilon_n)\rightarrow0,$$
contradicting (\ref{4.5}). The proof is complete.
\end{proof}

The following proposition provides  a necessary condition for the  tilt-stability of a proper lower semicontinuous function $f$ in terms of the $C^{1,\omega}$ smoothness of the concerned conjugate function.

\begin{proposition}
\label{pro-smooth}
Let $\omega:\mathbb{R}_+\rightarrow\mathbb{R}_+ $ be a function such that $\lim\limits_{t\rightarrow0^+}\omega(t)=\omega(0)=0$. Let  $f$ be a proper lower semicontinuous function on a Banach space $X$ and  $\bar x$ be a  minimizer of $f$. Suppose that there exist $r, \delta,\gamma \in(0,+\infty)$ and  a set-valued mapping $M: B_{X^*}(0,\delta)\rightrightarrows B_X[\bar x, r]$ with $\bar x \in M(0)$ such that
\begin{equation}
M(u^*)\subset \mathop{\arg\min}_{z\in B_X[\bar x,r]}f_{u^*}(z) \quad \forall u^*\in B_{X^*}(0,\delta)\label{pro-s1}
\end{equation}
and
\begin{equation}
M(x^*)\cap B_X(\bar x,\gamma)\subset M(u^*)+\omega (\|x^*-u^*\|)B_X \quad \forall x^*, u^*\in B_{X^*}(0,\delta).\label{pro-s0}
\end{equation}
Then, there exists $\delta'>0$ such that the conjugate function $(f+\delta_{B_X[\bar x, r]})^*$ is $C^{1,\omega}$ smooth on $B_{X^*}(0,\delta')$ and
\begin{equation}
 \{\triangledown(f+\delta_{B_X[\bar x, r]})^*(u^*)\}=M(u^*)\quad \forall u^* \in B_{X^*}(0, \delta').\label{pro-s2}
\end{equation}
\end{proposition}

\begin{proof}
 Let $u^*\in B_{X^*}(0,\delta)$ and $u\in M(u^*)$. Then, by (\ref{pro-s1}), one has
$$
(f+\delta_{B_X[\bar x, r]})^*(u^*)=\langle u^*, u\rangle -f(u),
$$
which implies that $u\in \partial (f+\delta_{B_X[\bar x, r]})^*(u^*)$. Hence $M(u^*)\subset \partial (f+\delta_{B_X[\bar x, r]})^*(u^*)$. Note that the subdifferential mapping $\partial (f+\delta_{B_X[\bar x, r]})^*$ is monotone (because the conjugate function $(f+\delta_{B_X[\bar x, r]})^*$ is always convex). Therefore, $M$ is also  monotone. Thus, by (\ref{pro-s0}) and Lemma 4.2, there exists $\delta'\in(0,\;\delta)$ such that $M$ is single-valued on $B_{X^*}(0,\delta')$. It follows from (\ref{pro-s0}) and the continuity of $\omega$ that $M$ is a norm-norm continuous selection of $\partial (f+\delta_{B_X[\bar x, r]})^*$ on $B_{X^*}(0,\delta')$. This and  \cite[Proposition 2.8]{phelps}  imply that the convex function $(f+\delta_{B_X[\bar x, r]})^*$ is Fr\'{e}chet differentiable on $B_{X^*}(0,\delta')$ and
$$
 \triangledown(f+\delta_{B_X[\bar x, r]})^*(u^*)=M(u^*)\quad \forall u^* \in B_{X^*}(0, \delta').
$$
The proof is complete.
\end{proof}

From Proposition 4.3 and Definition 1.1(ii), we have the following corollary.

\begin{corollary}
Let $\psi$ be an admissible function such that $\lim\limits_{t\rightarrow0^+}\psi(t)=\psi(0)$. Let $f$ be a proper lower semicontinuous function on a Banach space $X$ and $\bar x$ be a  minimizer of $f$. Then, the following statements are equivalent:\\
(i) $f$ has weak $\psi$-{\it TSLM} at $\bar x$, namely there exist $r,\gamma,\kappa, \delta, \tau\in(0,\;+\infty)$ such that
$$
\mathop{\arg\min}_{z\in B_X[\bar x,r]}f_{x^*}(z)\cap B_X(\bar x,\gamma)\subset\mathop{\arg\min}_{z\in B_X[\bar x,r]}f_{u^*}(z)+{\kappa} \psi (\tau \|x^*-u^*\|)B_X
$$
for all $x^*, u^*\in B_{X^*}(0,\delta)$.\\
(ii) There exist $\delta,r,\gamma,\kappa,\tau \in(0,+\infty)$ and  a set-valued mapping $M: B_{X^*}(0,\delta)\rightrightarrows B_X[\bar x, r]$ with $\bar x\in M(0) $ such that (\ref{pro-s1})  holds and $$
M(x^*)\cap B_X(\bar x,\gamma)\subset M(u^*)+\kappa\psi (\tau\|x^*-u^*\|)B_X \quad \forall x^*, u^*\in B_{X^*}(0,\delta).
$$
(iii) $f$ has $\psi$-{\it TSLM} at $\bar x$.
\end{corollary}

In the special case when $\varphi(t)=t^2$ and $\psi(t)=t$, recall that Drusvyatskiy and Lewis \cite{DL} proved that a proper lower semicontinuous function $f$ has $\varphi$-{\it SLWP} at $\bar x$ if and only if $f$ has $\psi$-{\it TSLM} at $\bar x$. In the case when $\varphi(t)=t^\frac{1+p}{p}$ and $\psi(t)=t^p$ with $p>0$, it was proved in a recent paper \cite{zheng siam 2014} that $f$ has $\varphi$-{\it SLWP} at $\bar x$ if and only if $f$ has $\psi$-{\it TSLM} at $\bar x$. For two general admissible functions $\varphi$ and $\psi$, it is interesting to determine a relationship between $\varphi$ and $\psi$ which makes the corresponding $\varphi$-{\it SLWP} and $\psi$-{\it TSLM} equivalent. This motivates us to make the following conjecture: {\it if $\varphi$ is a differentiable convex admissible function and $\psi$ is the inverse function $(\varphi')^{-1}$ of $\varphi'$ then $f$ has $\varphi$-{\it SLWP} at $\bar x$ if and only if $f$ has $\psi$-{\it TSLM} at $\bar x$.} With the help of Lemma 4.1 and refining the proof of \cite[Theorem 5.1]{zheng siam 2014}, we can establish the following result which prove the above conjecture.

\begin{theorem}
\label{theorem4}
Let $\varphi :\mathbb{R}_+\rightarrow \mathbb{R}_+$ be a differentiable and strictly convex admissible function with $\varphi'(0)=0$. Let $f$ be a proper lower semicontinuous function on a Banach space $X$. Then, $f$ has $\varphi$-{\it SLWP} at $\bar x\in {\rm dom} (f)$  if and only if $f$ has  $(\varphi')^{-1}$-{\it TSLM} at $\bar x$.

\end{theorem}

\begin{proof} First suppose that $f$ has $(\varphi')^{-1}$-{\it TSLM} at $\bar x$. Then there exist $\delta, r,\kappa,\tau \in(0,+\infty)$ and  $M: B_{X^*}(0,\delta)\rightarrow B_X[\bar x, r]$  with $M(0)=\bar x$  such that
\begin{equation}\label{theore3}
 M(u^*)\in\mathop{\arg\min}_{z\in B_X[\bar x,r]}f_{u^*}(z)\quad \forall u^*\in B_{X^*}(0,\delta)
\end{equation}
and
\begin{equation}\label{theore2}
\kappa\|M(x^*)-M(u^*)\| \leq (\varphi')^{-1}(\tau\|x^*-u^*\|)\quad \forall x^*, u^*\in B_{X^*}(0,\delta).
\end{equation}
Let $\omega(t):=\frac{1}{\kappa}(\varphi')^{-1}(\tau t)$ for all $t\in\mathbb{R}_+$. Then, since $\varphi$ is a differentiable and strictly convex admissible function with $\varphi'(0)=0$, $\omega$ is a continuous increasing function such that $\omega(0)=0$. Hence, by (\ref{theore3}), (\ref{theore2}) and Proposition 4.3, there exists  $\delta'>0$ such  that  $(f+\delta_{B_X[\bar x, r]})^*$ is $C^{1,\omega}$ smooth  on $B_{X^*}(0,\delta')$  and $\triangledown(f+\delta_{B_X[\bar x, r]})^*(u^*)=M(u^*)$  for all $u^*\in B_{X^*}(0,\delta')$; hence  $B_{X^*}(0,\delta')\subset{\rm dom}((f+\delta_{B_X[\bar x,r]})^*)$. Take $\delta_1>0$ such that
\begin{equation}\label{N4.12}
\delta_1+\omega^{-1}(2\omega(\delta_1))<\delta'\;\;{\rm and}\;\;r_0:=\omega(\delta_1)<r.
\end{equation}
Then, by Lemma 4.1 (applied to $E=X^*$ and $g=(f+\delta_{B_X[\bar x, r]})^*$), one has
\begin{eqnarray}\label{N4.13}
(f+\delta_{B_X[\bar x, r]})^{**}(x)&\geq& \langle x,u^*\rangle-(f+\delta_{B_X[\bar x, r]})^{*}(u^*)+ \int_0^{\|x-M(u^*)\|} \omega^{-1} (s)ds\nonumber\\
&=& \langle x,u^*\rangle-(f+\delta_{B_X[\bar x, r]})^{*}(u^*)+ \frac{1}{\tau}\int_0^{\|x-M(u^*)\|} \varphi'(\kappa s)ds\nonumber\\
&=&\langle x,u^*\rangle-(f+\delta_{B_X[\bar x, r]})^{*}(u^*)+ \frac{1}{\kappa\tau}\varphi(\kappa\|x-M(u^*)\|)
\end{eqnarray}
for all $(u^*,x)\in B_{X^*}(0, \delta_1)\times B_X(\bar x,r_0)$. By (\ref{theore2}), one has
$$\kappa\|\bar x-M(u^*)\|=\kappa\|M(0)-M(u^*)\|\leq(\varphi')^{-1}(\tau\|u^*\|)\quad\forall u^*\in B_{X^*}(0,\delta),$$
and so there exists $\delta_0\in(0,\;\min\{\delta,\delta_1\})$ such that $\|\bar x-M(u^*)\|<r_0$ for all $u^*\in B_{X^*}(0,\delta_0)$. It follows from (\ref{theore3}) and (\ref{N4.12}) that \begin{equation}\label{N4.14}
M(u^*)\in \mathop{\arg\min}_{z\in B_X(\bar x,r_0)}f_{u^*}(z)\quad \forall u^*\in B_{X^*}(0,\delta_0).
\end{equation}
Since $(f+\delta_{B_X[\bar x, r]})^{**}(x) \leq (f+\delta_{B_X[\bar x, r]})(x)=f(x)$ for all $x\in B_X[\bar x, r]$,
(\ref{N4.13}) and the choice of $\delta_0$ imply that
\begin{equation}
f_{u^*}(x)+(f+\delta_{B_X[\bar x, r]})^{*}(u^*) \geq \frac{1}{\kappa\tau}\varphi(\kappa\|x-M(u^*)\|)\label{theore6}
\end{equation}
for all $x\in B_X(\bar x, r_0)$ and $u^*\in B_{X^*}(0, \delta_0)$. Noting (by (\ref{theore3})) that
$$
(f+\delta_{B_X[\bar x, r]})^{*}(u^*)=\langle u^*, M(u^*)\rangle-f(M(u^*))=-f_{u^*}(M(u^*))\quad \forall u^*\in B_{X^*}(0, \delta_0),
$$
it follows that
$$f_{u^*}(x)-f_{u^*}(M(u^*)) \geq \frac{1}{\kappa\tau}\varphi(\kappa\|x-M(u^*)\|)\quad\forall (x,u^*)\in B_X(\bar x, r_0)\times B_{X^*}(0, \delta_0).$$
This and  (\ref{N4.14}) imply that $f$ has $\varphi$-{\it SLWP} at $\bar x$. This shows that sufficiency part holds.

To prove the necessity part, suppose that $f$ has $\varphi$-{\it SLWP} at $\bar x$, namely there exist $\delta,r,\kappa,\tau\in(0,+\infty)$ such that for any $u^*\in B_{X^*}(0,\delta)$ there exists $x_{u^*}\in B_X[\bar x, r]$, with $x_0=\bar x$,  satisfying
\begin{equation}
\varphi (\kappa\|x-x_{u^*}\|)\leq  \tau(f_{u^*}(x)-f_{u^*}(x_{u^*}))\quad \forall x \in B_X[\bar x, r].\label{propo:1}
\end{equation}
Let $u_1^*, u_2^*\in B_{X^*}(0, \delta)$; by  (\ref{propo:1}), one has
\begin{eqnarray}\label{N}
2\varphi (\kappa\|x_{u_2^*}-x_{u_1^*}\|)&\leq&  \tau(f_{u_1^*}(x_{u_2^*})-f_{u_1^*}(x_{u_1^*})+f_{u_2^*}(x_{u_1^*})-f_{u_2^*}(x_{u_2^*}))\nonumber\\
&=&\tau\langle u_1^*-u_2^*, x_{u_1^*}-x_{u_2^*}\rangle\nonumber\\
&\leq& \tau\|u_1^*-u_2^*\| \cdot \|x_{u_1^*}-x_{u_2^*}\|.
\end{eqnarray}
Since $\varphi$ is a differentiable and strictly convex admissible function with $\varphi'(0)=0$, $\varphi'$ is a nonnegative increasing function on $\mathbb{R}_+$. Hence
\begin{eqnarray*}
\varphi(\kappa\|x_{u_1^*}-x_{u_2^*}\|)&\geq& \varphi(\kappa\|x_{u_1^*}-x_{u_2^*}\|)-\varphi\left(\frac{\kappa\|x_{u_1^*}-x_{u_2^*}\|}{2}\right)\\
&=&\int_0^1\varphi'\left(\frac{\kappa\|x_{u_1^*}-x_{u_2^*}\|(1+t)}{2}\right)\frac{\kappa\|x_{u_1^*}-x_{u^*_2}\|}{2}dt\\
&\geq&\varphi'\left(\frac{\kappa\|x_{u_1^*}-x_{u_2^*}\|}{2}\right)\frac{\kappa\|x_{u_1^*}-x_{u_2^*}\|}{2}.
\end{eqnarray*}
This and (\ref{N}) imply that $\varphi'\left(\frac{\kappa\|x_{u_2^*}-x_{u_1^*}\|}{2}\right)\leq \frac{\tau}{\kappa}\|u_1^*-u_2^*\|$ for all $u_1^*, u_2^*\in B_{X^*}(0, \delta)$, that is,
$$\frac{\kappa}{2}\|x_{u_1^*}-x_{u_2^*}\| \leq (\varphi')^{-1}\big(\frac{\tau}{\kappa}\|u_1^*-u_2^*\|\big)\quad\forall u_1^*, u_2^*\in B_{X^*}(0, \delta).$$
Noting (by (\ref{propo:1})) that
$$
\mathop{\arg\min}_{x\in B_X[\bar x, r]}{f_{u^*}(x)}=\{ x_{u^*}\} \quad \forall u^*\in B_{X^*}(0,\delta),
$$
It follows that $f$ has  $(\varphi')^{-1}$-{\it TSLM} at $\bar x$.
The proof is complete.
\end{proof}

With the help of Theorem 4.5 and Proposition 4.3, we now can prove Theorem 3.4.

{\bf {\it Proof of Theorem 3.4}.} By Theorem 4.5,
the  $\varphi$-{\it SLWP} assumption means that $f$ has $(\varphi')^{-1}$-{\it TSLM} at $\bar x$. Hence there exist $\delta,r,\kappa,\tau\in(0,\;+\infty)$ and a mapping $M:B_{X^*}(0,\delta)\rightarrow B_X[\bar x,r]$ with $M(0)=\bar x$ such that
$$
M(u^*)\in\mathop{\arg\min}_{x\in B_X[\bar x, r]}{f_{u^*}(x)}\;\;{\rm and}\;\;\kappa\|M(u^*)-M(v^*)\| \leq (\varphi')^{-1}(\tau\|u^*-v^*\|)$$
for all  $u^*,v^*\in B_{X^*}(0,\delta)$.
Let $h:=\overline{\rm co}(f+\delta_{B_X[\bar x, r]})$. Then, $h$ is a proper lower semicontinuous convex function, $h^*=(f+\delta_{B_X[\bar x, r]})^*$, and it follows from Proposition 4.3 (applied to $\omega(t)=\frac{1}{\kappa}(\varphi')^{-1}(\tau t)$) that there exists $\delta'\in(0,\;\delta)$ such that $h^*$ is smooth on $B_{X^*}(0,\delta')$, $\triangledown h^*(0)=x_0=\bar x$ and
\begin{equation}\label{4.18}
\|\triangledown h^*(v^*)-\triangledown h^*(u^*)\| \leq \frac{1}{\kappa}(\varphi')^{-1}(\tau\|v^*-u^*\|)\quad\forall v^*, u^*\in B_{X^*}(0,\delta').
\end{equation}
Hence
\begin{equation}\label{4.19}
\|\triangledown h^*(u^*)-\bar x\|\leq \frac{1}{\kappa}(\varphi')^{-1}(\tau\|u^*\|)\quad\forall u^*\in B_{X^*}(0,\delta').
\end{equation}
Note (by the convexity of $h$) that  $u^*\in \partial h(x)$ if and only if $x\in\partial h^*(u^*)$. One has
\begin{equation}\label{4.20}
(\partial h)^{-1}(u^*)=\{\triangledown h^*(u^*)\}\quad\forall u^*\in B_{X^*}(0,\delta').
\end{equation}
Thus, to complete the proof, it suffices to show that there exists $\kappa'>0$ such that
\begin{equation}\label{N4.21}
\;\;\varphi'(\kappa'\|x-\triangledown h^*(u^*)\|)\leq\tau d(u^*,\partial h(x))\quad\forall (x,u^*)\in B_X(\bar x,\frac{\delta'}{2})\times B_{X^*}(0,\frac{\delta'}{2}).
\end{equation}
Let $(x,u^*)\in B_X(\bar x,\frac{\delta'}{2})\times B_{X^*}(0,\frac{\delta'}{2})$.
Then, by (\ref{4.18})--(\ref{4.20}),
$$\|x-\triangledown h^*(u^*)\|\leq \|x-\bar x\|+\frac{1}{\kappa}(\varphi')^{-1}(\tau\|u^*\|)\leq\frac{\delta'}{2}+\frac{1}{\kappa}(\varphi')^{-1}(\frac{\tau\delta'}{2}),$$
$$\|x-\triangledown h^*(u^*)\|=\|\triangledown h^*(x^*)-\triangledown h^*(u^*)\|\leq\frac{1}{\kappa}(\varphi')^{-1}(\tau\|x^*-u^*\|)\;\;\forall x^*\in\partial h(x)\cap B_{X^*}(0,\delta')$$
and so
$$\|x-\triangledown h^*(u^*)\|\leq \frac{1}{\kappa}(\varphi')^{-1}(\tau d(u^*,\partial h(x)\cap B_{X^*}(0,\delta'))).$$
Therefore,
\begin{equation}\label{4.20'}
\quad \quad\|x-\triangledown h^*(u^*)\|\leq \frac{1}{\kappa}\min\{(\varphi')^{-1}(\tau d(u^*,\partial h(x)\cap B_{X^*}(0,\delta'))),\frac{\kappa\delta'}{2}+(\varphi')^{-1}(\frac{\tau\delta'}{2})\}.
\end{equation}
Since $d(u^*,\partial h(x)\cap (X^*\setminus B_{X^*}(0,\delta')))\geq d(u^*,X^*\setminus B_{X^*}(0,\delta'))\geq\frac{\delta'}{2}$,
$$d(u^*,\partial h(x))\geq \min\{d(u^*,\partial h(x)\cap B_{X^*}(0,\delta')),\frac{\delta'}{2}\}.$$
Hence
$$(\varphi')^{-1}(\tau d(u^*,\partial h(x)))\geq\min\{(\varphi')^{-1}(\tau d(u^*,\partial h(x)\cap B_{X^*}(0,\delta'))),(\varphi')^{-1}(\frac{\tau\delta'}{2})\}.$$
Letting $\beta:=\frac{\frac{\kappa\delta'}{2}+(\varphi')^{-1}(\frac{\tau \delta'}{2})}{(\varphi')^{-1}(\frac{\tau\delta'}{2})}$, it follows from (\ref{4.20'}) that
$$\|x-\triangledown h^*(u^*)\|\leq \frac{\beta}{\kappa}(\varphi')^{-1}(\tau d(u^*,\partial h(x))),$$
that is,
$$\varphi'\left(\frac{\kappa}{\beta}\|x-\triangledown h^*(u^*)\|\right)\leq {\tau}d(u^*,\partial h(x)).$$
This shows  that (\ref{N4.21}) holds with $\kappa'=\frac{\kappa}{\beta}$. The proof is complete.

\section{Stable weak well-posedness}
If a proper lower semicontinuous function $f$ has the  stable well-posedness at $\bar x$, then there exist $r,\delta\in(0,\;+\infty)$ such that
$\mathop{\arg\min}_{x\in B_X[\bar x, r]}{f_{u^*}(x)}$ is a singleton for any $u^*\in B_{X^*}(0,\delta)$. It is natural to consider the case when $\mathop{\arg\min}_{x\in B_X[\bar x, r]}{f_{u^*}(x)}$ is not a singleton. This yields the following notion.
\begin{definition}
Let $\varphi:\mathbb{R}_+\rightarrow\mathbb{R}_+$ be an admissible function and let  $f$ be a proper lower semicontinuous extended real-valued function on a Banach space $X$.  We say that $f$ has stable weak local well-posedness at $\bar x \in {\rm dom} (f)$ with respect to $\varphi$ (in brief, $\varphi$-{\it SWLWP}) if there exist $r, \gamma,\delta,\tau,\kappa\in(0,+\infty)$ such that $\min\limits_{x\in B_X[\bar x,r]}f(x)=f(\bar x)$ and
\begin{equation}\label{5.17}
\varphi\big(\tau d\big(x,\mathop{\arg\min}_{x\in B_X[\bar x, r]}{f_{u^*}(x)}\big)\big)\leq \kappa(f_{u^*}(x)-\min\limits_{z\in B_{X}[\bar x,r]}f_{u^*}(z))
\end{equation}
for all $(x,u^*)\in B_X(\bar x,\gamma)\times B_{X^*}(0,\delta)$.
\end{definition}\\

Given an increasing admissible function $\varphi$, it is clear that the corresponding well-posedness implies the weak well-posedness. The following example shows that the converse implication is not true. Let $f:\mathbb{R}\rightarrow\mathbb{R}$ be such that $f(t)=0$ for all $t\in(-\infty,0]$ and $f(t)=\varphi(t)$ for all $t\in(0,\;+\infty)$. Then $$\mathop{\arg\min}_{t\in\mathbb{R}}f(t)=(-\infty,\;0]\;\;{\rm  and}\;\; \varphi(d(x,\mathop{\arg\min}_{t\in\mathbb{R}}f(t))=f(x)-\min\limits_{t\in\mathbb{R}}f(t)\;\;\forall x\in\mathbb{R}.$$
Hence, $f$ has the weak well-posedness but does not have the well-posedness because $\mathop{\arg\min}_{t\in\mathbb{R}}f(t)$ is not a singleton. Nevertheless, the following theorem shows that the corresponding stable well-posedness and stable weak well-posedness are equivalent when $f$ undergoes small tilt perturbations, which was proved by Zheng and Ng \cite{zheng nonlinear 2014} in the case when $\varphi(t)=t^q$.

\begin{theorem}
\label{stable tilt}
Let $\varphi:\mathbb{R}_+\rightarrow\mathbb{R}_+$ be a differentiable and strictly convex admissible function such that $\varphi'(0)=0$. Let $X$  be a Banach space and $f:X\rightarrow\mathbb{R}\cup\{+\infty\}$ be a proper lower semicontinuous function. Then, $f$ has $\varphi$-{\it SLWP} at $\bar x\in {\rm dom} (f)$ if and only if $f$ has $\varphi$-{\it SWLWP} at $\bar x$.
\end{theorem}
\begin{proof}
The necessity part is trivial. For  the sufficiency part,  suppose that $f$ has $\varphi$-{\it SWLWP} at $\bar x$. Then there exist $r, \gamma, \delta,\tau,\kappa\in(0,+\infty)$ such that $\min\limits_{x\in B_X[\bar x,r]}f(x)=f(\bar x)$ and (\ref{5.17}) holds. Letting
$$\mathcal{M}(u^*):=\mathop{\arg\min}_{z\in B_X[\bar x,r]}f_{u^*}(z)\quad\forall u^*\in X^*,$$
it suffices to show that  $\mathcal{M}(u^*)$ is a singleton for each  $u^*\in X^*$ close to 0. Let $\gamma':=\frac{1}{4}\min\{r,\gamma\}$ and $\delta':=\min\{\delta,\frac{\varphi(\tau \gamma')}{\kappa r}\}$. Then, by Proposition 4.3, we only need to show that there exists  a continuous function $\omega:\mathbb{R}_+\rightarrow\mathbb{R}_+$ with $\omega(0)=0$ such that
\begin{equation}\label{5.19}
\mathcal{M}(u^*)\cap B_X(\bar x,\gamma')\subset \mathcal{M}(v^*)+\omega (\|u^*-v^*\|)B_X\quad\forall u^*,v^*\in B_{X^*}(0,\delta').
\end{equation}
By (\ref{5.17}), one has
\begin{eqnarray*}
\varphi \left(\tau d\left(\bar x,\mathcal{M}(u^*)\right)\right)&\leq&  \kappa(f_{u^*}(\bar x)-\min_{z\in B_X[\bar x, r]}f_{u^*}(z))\\
&=&\kappa(\min_{z\in B_X[\bar x, r]}f(z)-\min_{z\in B_X[\bar x, r]}(f(z)-\langle u^*,z-\bar x\rangle))\\
&\leq& \kappa(\min_{z\in B_X[\bar x, r]}f(z)-\min_{z\in B_X[\bar x, r]}(f(z)-\|u^*\|r))\\
&=&\kappa\|u^*\|r<\varphi(\tau \gamma')
\end{eqnarray*}
for all $u^*\in B_{X^*}(0,\delta')$. Since a strictly convex admissible function is increasing,
$$d\left(\bar x,\mathcal{M}(u^*)\right)<\gamma'\quad\forall u^*\in B_{X^*}(0,\delta'),$$
namely, for any $u^*\in B_{X^*}(0,\delta')$ there exists $x_{u^*}\in \mathcal{M}(u^*)$ such that
\begin{equation}\label{theo:3}
\|x_{u^*}-\bar x\|<\gamma'.
\end{equation}
Let $u^*, v^*\in B_{X^*}(0,\delta')$ and $u\in \mathcal{M}(u^*)\cap B_X(\bar x, \gamma')$, and take a sequence $\{v_n\}$ in $\mathcal{M}(v^*)$ such that
\begin{equation}
\lim\limits_{n\rightarrow\infty}\|u-v_n\|=d\left(u,\mathcal{M}(v^*)\right).\label{theo:4}
\end{equation}
Noting (by (\ref{theo:3})) that
$$d\left(u,\mathcal{M}(v^*)\right)\leq\|u-x_{v^*}\|\leq\|u-\bar x\|+\|\bar x-x_{v^*}\|<2\gamma',$$
we can assume without loss of generality that $\|u-v_n\|<2\gamma'$ for all $n\in\mathbb{N}$, and so
$$\|v_n-\bar x\|\leq\|v_n-u\|+\|u-\bar x\|<3\gamma'<\gamma\quad\forall n\in\mathbb{N}.$$
Thus,
by (\ref{5.17}), one has
\begin{eqnarray*}
\varphi \left(\tau d\left(u,\mathcal{M}(v^*)\right)\right)&\leq&\kappa(f_{v^*}(u)-\min\limits_{z\in B_X[\bar x,r]}f_{v^*}(z))\\
&=&  \kappa(f_{v^*}(u)-f_{v^*}(v_n))
\end{eqnarray*}
and
$$
\varphi\left(\tau d\left(v_n,\mathcal{M}(u^*)\right)\right)\leq  \kappa(f_{u^*}(v_n)-f_{u^*}(u))
$$
for all $n\in\mathbb{N}$. Therefore,
\begin{eqnarray*}
\varphi \left(\tau d\left(u,\mathcal{M}(v^*)\right)\right)&\leq&
\varphi \left(\tau d\left(u,\mathcal{M}(v^*)\right)\right)+
\varphi\left(\tau d\left(v_n,\mathcal{M}(u^*)\right)\right)\\
&\leq&\kappa(f_{v^*}(u)-f_{v^*}(v_n)+f_{u^*}(v_n)-f_{u^*}(u))\\
&=&\kappa\langle u^*-v^*,u-v_n\rangle\\
&\leq&\kappa\|u^*-v^*\|\|u-v_n\|
\end{eqnarray*}
for all $n\in\mathbb{N}$. This and (\ref{theo:4}) imply  that
$$\varphi \left(\tau d\left(u,\mathcal{M}(v^*)\right)\right)\leq \kappa \|u^*-v^*\|d\left(u,\mathcal{M}(v^*)\right).$$
Noting that $\varphi(t)\geq\varphi(t)-\varphi(\frac{t}{2})\geq\varphi'(\frac{t}{2})\frac{t}{2}$ for all $t\in\mathbb{R}_+$, it follows that
$$\varphi'\left(\frac{\tau}{2}d\left(u,\mathcal{M}(v^*)\right)\right)\leq \frac{2\kappa}{\tau} \|u^*-v^*\|,$$
that is,
$$d\left(u,\mathcal{M}(v^*)\right)\leq\frac{2}{\tau}(\varphi')^{-1}(\frac{2\kappa}{\tau} \|u^*-v^*\|).$$
This implies that
$$u\in \mathcal{M}(v^*)+\frac{3}{\tau}(\varphi')^{-1}(\frac{2\kappa}{\tau} \|u^*-v^*\|)B_X.$$
Hence
$$\mathcal{M}(u^*)\cap B_X(\bar x,\gamma')\subset \mathcal{M}(v^*)+\frac{3}{\tau}(\varphi')^{-1}(\frac{2\kappa}{\tau} \|u^*-v^*\|)B_X.$$
This shows that (\ref{5.19}) holds with $\omega(t)=\frac{3}{\tau}(\varphi')^{-1}(\frac{2\kappa t}{\tau})$. The proof is complete.
\end{proof}

\begin{corollary}
Let $\varphi:\mathbb{R}_+\rightarrow\mathbb{R}_+$ be a differentiable and strictly convex admissible function such that $\varphi'(0)=0$. Let $X$  be a Banach space and $f:X\rightarrow\mathbb{R}\cup\{+\infty\}$ be a lower semicontinuous function with $\bar x\in{\rm dom}(f)$. Consider the following statements:\\
 (i) $f$ has $\varphi$-{\it SLWP} at $\bar x$.\\
 (ii) $f$ has $\varphi$-{\it SWLWP} at $\bar x$.\\
 (iii) $f$ has  $(\varphi')^{-1}$-{\it TSLM} at $\bar x$.\\
 (iv) $f$ has weak $(\varphi')^{-1}$-{\it TSLM} at $\bar x$.\\
 (v) $\partial f$ is strongly metrically $\varphi'$-regular at $(\bar x, 0 )$.\\
 (vi) $\partial f$ is metrically $\varphi'$-regular at $(\bar x,0)$.\\
Then, $(i)\Leftrightarrow (ii)\Leftrightarrow (iii)\Leftrightarrow(iv)\Leftarrow (v)\Rightarrow (vi)$. If, in addition, $f$ is convex, then $(i)\Leftrightarrow (ii)\Leftrightarrow (iii)\Leftrightarrow(iv)\Leftrightarrow (v)\Leftrightarrow (vi)$.
\end{corollary}
\begin{proof}
$(i)\Leftrightarrow (ii)\Leftrightarrow (iii)\Leftrightarrow(iv)\Leftarrow (v)$ are immediate from Theorems 5.2, 4.5 and 3.4 and  Corollary 4.4, while (v)$\Rightarrow$(vi) is trivial.

Now suppose that $f$ is convex. Since (i)$\Rightarrow$(v) is immediate from Corollary 3.5, it suffices  to show (vi)$\Rightarrow$(v).  By (vi), take $\tau,\kappa,r\in(0,\;+\infty)$ such that
\begin{equation}\label{4.6n}
\varphi'(\tau d(x,(\partial f)^{-1}(x^*)))\leq\kappa d(x^*,\partial f(x))\quad\forall (x,x^*)\in B_X(\bar x,r)\times B_{X^*}(0,r).
\end{equation}
Thus,  by Lemma 4.2 (applied to $F=(\partial f)^{-1}$), we only need  to show that there exist $\gamma,\delta\in(0,\;+\infty)$ such that
$$(\partial f)^{-1}(x^*)\cap B_X(\bar x,\gamma)\subset(\partial f)^{-1}(u^*)+\omega(\|x^*-u^*\|)B_X\quad\forall x^*,u^*\in B_{X^*}(0,\delta),$$
where $\omega(t)=\frac{2}{\tau}(\varphi')^{-1}(\kappa t)$ for all $t\in\mathbb{R}_+$. To do this, suppose to the contrary that there exists a sequence $(x_n^*,u_n^*,x_n)\longrightarrow(0,0,\bar x)$ such that
$$x_n\in(\partial f)^{-1}(x_n^*)\;\;{\rm and}\;\;x_n\not\in(\partial f)^{-1}(u_n^*)+\omega(\|x_n^*-u_n^*\|)B_X\quad\forall n\in\mathbb{N}.$$
It follows from  (\ref{4.6n}) that
$$\varphi'(\tau d(x_n,(\partial f)^{-1}(u_n^*)))\leq \kappa d(u_n^*,\partial f(x_n))\leq\kappa\|u_n^*-x_n^*\|,$$
and so
$d(x_n,(\partial f)^{-1}(u_n^*))\leq \frac{1}{\tau}(\varphi')^{-1}(\kappa\|u_n^*-x_n^*\|)$ for all sufficiently large $n$. This and the definition of $\omega$ imply that
$x_n\in(\partial f)^{-1}(u_n^*)+\omega(\|u_n^*-x_n^*\|)B_X$ for all sufficiently large $n$, a contradiction.   The proof is complete.
\end{proof}

\section{Second order condition}
In this section, in the finite dimension setting, we provide a sufficient condition for stable well-posedness in terms of the second subdifferential. Throughout this section,  $f$ is a proper lower semicontinuous function on $\mathbb{R}^n$; let $\partial f$ denote Mordukhovich's limiting subdifferential of $f$ and $\tilde{N}(\partial f,\cdot)$ denote Mordukhovich's limiting normal cone of $\partial f$ (see \cite{M} for its detail). For $(x,v)\in{\rm gph}(\partial f)$, adopting Mordukhovich's construction,  the second subdifferential $\partial^2f(x,v)$ of $f$ at $(x,v)$ is defined as
$$\partial^2 f(x,v)(h)=\{z\in \mathbb{R}^n:\;(z,-h)\in \tilde N({\rm gph}(\partial f), (x,v))\}\quad\forall h\in\mathbb{R}^n$$
(see \cite[Definition 2.2]{MN2}). For a convex admissible function $\psi$, let
$$\eta_{\psi}(x,v)(h):=\psi'_+(d(x,(\partial f)^{-1}(v-h)))$$
for all $(x,v,h)\in{\rm gph}(\partial f)\times\mathbb{R}^n$.

\begin{proposition}
Let $\psi$ be a convex admissible function and let $(\bar x,0)\in{\rm gph}(\partial f)$.  Suppose that ${\rm gph}(\partial f)$ is closed and that there exist $\kappa,r\in(0,\;+\infty)$ such that
\begin{equation}\label{6.1}
\kappa\|h\|^2\eta_{\psi}(x,v)(h)\leq\langle z,h\rangle
\end{equation}
for all $(x,v,h)\in({\rm gph}(\partial f)\times\mathbb{R}^n)\cap (B(\bar x,r)\times B(0,r)\times B(0,r))$ and $z\in\partial^2f(x,v)(h)$. Then $\partial f$ is metrically $\psi$-regular at $(\bar x,0)$.
\end{proposition}

\begin{proof}
First we show that there exist $\kappa_1,\tau_1,r_1\in(0,\;+\infty)$ such that
\begin{equation}\label{6.1'}
\psi(\kappa_1 d(x,(\partial f)^{-1}(v))\leq\tau_1d(v,\partial f(x))
\end{equation}
for all $(x,v)\in B(\bar x,r_1)\times(\partial f(B(\bar x,r_1))\cap B(0,r_1))$.
To do this, suppose to the contrary that there exists a sequence $\{(u_i, x_i,v_i)\}\subset \mathbb{R}^n\times\mathbb{R}^n\times\mathbb{R}^n$ such that $(u_i, x_i, v_i)\rightarrow(\bar x,\bar x, 0)$,
$$ v_i\in\partial f(u_i)\;\;{\rm and}\;\;\psi\big(\frac{1}{i}d(x_i,(\partial f)^{-1}(v_i))\big)>id(v_i,\partial f(x_i))\quad\forall i\in\mathbb{N}.$$
Thus,
\begin{equation}\label{6.2'}
0<d(x_i,(\partial f)^{-1}(v_i))\leq\|x_i-u_i\|\rightarrow0,
\end{equation}
and there exists $y_i\in\partial f(x_i)$ such that
\begin{equation}\label{6.3'}
\|v_i-y_i\|<\frac{1}{i}\psi\big(\frac{1}{i}d(x_i,(\partial f)^{-1}(v_i))\big)\leq\frac{1}{i}\psi\big(\frac{1}{i}\|x_i-u_i\|)\rightarrow0.
\end{equation}
Define
$$g_{i}(u,v):=\|v-v_i\|+\delta_{{\rm gph}(\partial f)}(u,v)\quad\forall (u,v)\in \mathbb{R}^n\times\mathbb{R}^n.$$
Then, $g_i$ is lower semicontinuous,  and
$$g_{i}(x_i,y_i)<\inf\limits_{(u,v)\in\mathbb{R}^n\times\mathbb{R}^n}g_{i}(u,v)+\frac{1}{i}\psi\big(\frac{1}{i}d(x_i,(\partial f)^{-1}(v_i))\big).$$
For any $j\in\mathbb{N}$, letting
$$\|(u,v)\|_j:=\|u\|+\frac{1}{j}\|v\|\quad\forall (u,v)\in\mathbb{R}^n\times\mathbb{R}^n,$$
it follows from the Ekeland variational principle that  there exists $(x_{ij},y_{ij})\in{\rm gph}(\partial f)$ such that
\begin{equation}\label{6.2}
\|(x_{ij},y_{ij})-(x_i,y_i)\|_j<\frac{1}{i}d(x_i,(\partial f)^{-1}(v_i)),
\end{equation}
\begin{equation}\label{6.3}
\|y_{ij}-v_i\|=g_{i}(x_{ij},y_{ij})\leq g_{i}(x_i,y_i)=\|y_i-v_i\|
\end{equation}
and
\begin{equation}\label{6.6'}
g_{i}(x_{ij},y_{ij})\leq g_{i}(u,v)+\frac{\psi\big(\frac{1}{i}d(x_i,(\partial f)^{-1}(v_i))\big)}{d(x_i,(\partial f)^{-1}(v_i))}\|(u,v)-(x_{ij},y_{ij})\|_j
\end{equation}
for all $(u,v)\in\mathbb{R}^n\times\mathbb{R}^n$.
Clearly, (\ref{6.2}) and (\ref{6.3}) imply that $\{(x_{ij},y_{ij}\}_{j\in\mathbb{N}}$ is a bounded sequence in $\mathbb{R}^n\times\mathbb{R}^n$. Without loss of generality, we can assume that $(x_{ij},y_{ij})\rightarrow(\bar x_i,\bar v_i)\in{\rm gph}(\partial f)$ as $j\rightarrow\infty$ (passing to a subsequence if necessary). It follows from (\ref{6.2})---(\ref{6.6'}) that
$$
\|\bar x_i-x_i\|\leq \frac{1}{i}d(x_i,(\partial f)^{-1}(v_i)),\;\;\|\bar v_i-v_i\|\leq\|y_i-v_i\|
$$
and
\begin{equation}\label{6.7}
\|\bar v_i-v_i\|\leq\|v-v_i\|+\delta_{{\rm gph}(\partial f)}(u,v)+\frac{\psi\big(\frac{1}{i}d(x_i,(\partial f)^{-1}(v_i))\big)}{d(x_i,(\partial f)^{-1}(v_i))}\|u-\bar x_i\|
\end{equation}
for all $(u,v)\in\mathbb{R}^n\times\mathbb{R}^n$. Hence, by (\ref{6.2'}), (\ref{6.3'}) and $(x_i,v_i)\rightarrow(\bar x,0)$, one has
\begin{equation}\label{6.9n}
0<d(x_i,(\partial f)^{-1}(v_i)))\leq\frac{i}{i-1}d(\bar x_i,(\partial f)^{-1}(v_i))
\end{equation}
and
$$\bar v_i\not=v_i\;\;{\rm and}\;\;(\bar x_i,\bar v_i)\rightarrow(\bar x,0).$$
It follows from (\ref{6.9n}) and the convexity of $\psi$ that
$$0<\frac{\psi\big(\frac{1}{i}d(x_i,(\partial f)^{-1}(v_i))\big)}{\frac{1}{i}d(x_i,(\partial f)^{-1}(v_i))}
\leq \frac{\psi\big(\frac{1}{i-1}d(\bar x_i,(\partial f)^{-1}(v_i))\big)}{\frac{1}{i-1}d(\bar x_i,(\partial f)^{-1}(v_i))}
\leq\psi'_+(d(\bar x_i,(\partial f)^{-1}(v_i)))$$
for all $i>1$. This and  (\ref{6.7}) imply that
$$\|\bar v_i-v_i\|\leq\|v-v_i\|+\delta_{{\rm gph}(\partial f)}(u,v)+\frac{1}{i}\psi'_+\big(d(\bar x_i,(\partial f)^{-1}(v_i))\big)\|u-\bar x_i\|$$
for all $(u,v)\in\mathbb{R}^n\times\mathbb{R}^n$. Hence,
\begin{eqnarray*}
(0,0)&\in&\{0\}\times\partial\|\cdot-v_i\|(\bar v_i)+\partial\delta_{{\rm gph}(\partial f)}(\bar x_i,\bar v_i)+\frac{1}{i}\psi'_+\big(d(\bar x_i,(\partial f)^{-1}(v_i))\big)B_{\mathbb{R}^n}\times\{0\}\\
&\subset&\{0\}\times\big\{\frac{\bar v_i-v_i}{\|\bar v_i-v_i\|}\big\}+\tilde N({\rm gph}(\partial f),(\bar x_i,\bar v_i))+\frac{1}{i}\psi'_+\big(d(\bar x_i,(\partial f)^{-1}(v_i))\big)B_{\mathbb{R}^n}\times\{0\},
\end{eqnarray*}
and so there exists $x_i^*\in B_{\mathbb{R}^n}$ such that
$$\frac{1}{i}\psi'_+\big(d(\bar x_i,(\partial f)^{-1}(v_i))\big)x_i^*\in \partial^2 f(\bar x_i,\bar v_i)\left(\frac{\bar v_i-v_i}{\|\bar v_i-v_i\|}\right).$$
Let $h_i:=\bar v_i-v_i$. Then, $v_i=\bar v_i-h_i$,
$$z_i:=\frac{1}{i}\|h_i\|\eta_{\psi}(\bar x_i,\bar v_i)(h_i)x_i^*=\frac{1}{i}\|h_i\|\psi'_+\big(d(\bar x_i,(\partial f)^{-1}(\bar v_i-h_i))\big)x_i^*\in\partial^2f(\bar x_i,\bar v_i)(h_i)$$
and so
$$
\langle z_i,h_i\rangle=\frac{1}{i}\|h_i\|\eta_{\psi}(\bar x_i,\bar v_i)(h_i)\langle x_i^*,h_i\rangle
\leq\frac{1}{i}\|h_i\|^2\eta_{\psi}(\bar x_i,\bar v_i)(h_i).
$$
Noting that $0<\psi'_+\big(d(\bar x_i,(\partial f)^{-1}(\bar v_i-h_i))$, it follows from (\ref{6.1})  that $\kappa\leq\frac{1}{i}$ for all sufficiently large $i$, a contradiction. Therefore, there exist $\kappa_1,\tau_1,r_1\in(0,\;+\infty)$ such that
(\ref{6.1'}) holds for all $(u,v)\in B(\bar x,r_1)\times(\partial f(B(\bar x,r_1))\cap B(0,r_1))$. Let $r_2\in(0,\;r_1)$. We claim that there exists $\delta\in(0,\;r_2)$ such that   $B(0,\delta)\subset \partial f(B[\bar x,r_2])$. Granting this, one has
$$B(\bar x,\delta)\times B(0,\delta)\subset B(\bar x,r_1)\times (\partial f(B(\bar x,r_1)\cap B(0,r_1)).$$
This  and (\ref{6.1'}) imply that $\partial f$ is metrically $\psi$-regular at $(\bar x,0)$. It remains to show that there exists $\delta\in(0,\;r_2)$ such that   $B(0,\delta)\subset \partial f(B[\bar x,r_2])$.
Indeed, if this is not the case, there exists a sequence $\{y_k\}$ converging to 0 such that each $y_k\not\in\partial f(B[\bar x,r_2])$.
Noting that $\partial f(B[\bar x,r_2])$ is closed (thanks to the compactness of  $B[\bar x,r_2]$ and the closedness of ${\rm gph}(\partial f)$), there exists $w_k\in\partial f(B[\bar x,r_2])$ such that
\begin{equation}\label{6.9'}
0<\|y_k-w_k\|=d(y_k,\partial f(B[\bar x,r_2]))\leq \|y_k\|\rightarrow0,
\end{equation}
and so $w_k\rightarrow0$. It follow from (\ref{6.1'}) that
$$\psi(\kappa_1d(\bar x,(\partial f)^{-1}(w_k)))\leq\tau_1d(w_k,\partial f(\bar x))\leq\tau_1\|w_k\|\rightarrow0$$
Hence, $\kappa_1d(\bar x,(\partial f)^{-1}(w_k))\rightarrow 0$ and so
there exists $a_k\in(\partial f)^{-1}(w_k)$ such that $a_k\rightarrow\bar x$. On the other hand, the equality of (\ref{6.9'}) means
$$\langle y_k-w_k,y-w_k\rangle\leq\frac{1}{2}\|y-w_k\|^2\quad\forall y\in\partial f(B[\bar x,r_2]).$$
Hence
$$\langle (0,y_k-w_k),(x,y)-(a_k,w_k)\rangle\leq\frac{1}{2}\|(x,y)-(a_k,w_k)\|^2$$
for all $(x,y)\in {\rm gph}(\partial f)\cap (B[\bar x,r_2]\times\mathbb{R}^n)$. Since $(a_k,w_k)$ is an interior point of $B[\bar x,r_2]\times\mathbb{R}^n$ for all $k$ large enough, $(0,y_k-w_k)\in \tilde N({\rm gph}(\partial f),(a_k,w_k))$, namely
$0\in\partial^2f(a_k,w_k)(w_k-y_k)$. It follows from (\ref{6.1}) that
$$\kappa\|y_k-w_k\|^2\eta_{\psi}(a_k,w_k)(w_k-y_k)\leq\langle 0,y_k-w_k\rangle=0.$$
By the first inequality of (\ref{6.9'}), one has
$$0=\eta_{\psi}(a_k,w_k)(w_k-y_k)=\psi'_+(d(a_k,(\partial f)^{-1}(y_k))).$$
This and (\ref{3.1}) imply that  $d(a_k,(\partial f)^{-1}(y_k))=0$, and so $y_k\in\partial f(a_k)$. This contradicts that $a_k\rightarrow\bar x$ and $y_k\not\in\partial f(B[\bar x,r_2])$. The proof is complete.
\end{proof}

Note that  $\partial f$ is closed if $f$ is convex or continuous.
The following corollary is immediate from Proposition 6.1 and Corollary 5.3.
\begin{corollary}
Let $\psi$ be a convex admissible function and let $f$ be a proper lower semicontinuous convex function on $\mathbb{R}^n$. Let $\bar x$ be a minimizer of $f$  and suppose that there exist $\kappa,r\in(0,\;+\infty)$ such that (\ref{6.1}) holds
for all $(x,v,h)\in({\rm gph}(\partial f)\times\mathbb{R}^n)\cap (B(\bar x,r)\times B(0,r)\times B(0,r))$ and $z\in\partial^2f(x,v)(h)$.
Then, $f$ has $\varphi$-{\it SLWP} at $\bar x$ with $\varphi(t):=\int\limits_{0}^t\psi(t)dt$.
\end{corollary}

In the special case when $\psi(t)=t$, $\eta_{\psi}(x,v)(h)\equiv1$, and so (\ref{6.1}) means the positive definiteness of $\partial^2f(x,v)$. It is worth mentioning that under the assumption that $f$ is prox-regular and subdifferentially continuous at $(\bar x,0)$,  the positive definiteness of $\partial^2 f(\bar x,0)$ is equivalent to that $\bar x$ is a stable second order local minimizer of $f$ (cf. \cite{DMN,MN1,MN2,poliquin and rockafellar}). In the finite dimension setting, we note that the positive definiteness of $\partial^2 f(\bar x,0)$ is equivalent to  the positive definiteness of $\partial^2 f(x,v)$ for all $(x,v)\in{\rm gph}(\partial f)$ close to $(\bar x,0)$.
We conclude with the following questions:\\
1) Under some assumption similar to the prox-regularity and subdifferential continuity, does ``generalized positive definiteness" in the sense of (\ref{6.1}) imply that $\partial f$ is strongly metrically $\psi$-regular at $(\bar x,0)$ ?\\
2) If $f$ is a proper lower semicontinuous convex  function, is (\ref{6.1}) a necessary condition for $f$ to have $\varphi$-{\it SLWP} at $\bar x$ with  $\varphi(t)=\int\limits_{0}^t\psi(t)$?

\end{document}